\documentclass[leqno,onefignum,onetabnum]{siamltex1213}
\usepackage{amsfonts,amsmath,amssymb}
\usepackage{changes}
\usepackage{booktabs,ctable,threeparttable,multirow}
\usepackage{graphicx,boxedminipage,appendix}
\usepackage{bm,enumerate,algorithm,algorithmic}

\usepackage[caption=false]{subfig}
\usepackage{mathrsfs}

\usepackage{amsthm}
\usepackage{lineno}
\usepackage{enumerate,enumitem}
\usepackage{soul}

\usepackage{hyperref}
\usepackage{color,xcolor}

\usepackage{mathptmx}      
\RequirePackage{fix-cm}
\usepackage{latexsym}

\allowdisplaybreaks[4]

\newtheorem{theoremmy}{Theorem}[section]
\newtheorem{lemmamy}[theoremmy]{Lemma}

\newtheorem{exper}[theoremmy]{Experiment}
\newtheorem{remark}{Remark}
\numberwithin{equation}{section}

\newcommand{\mi}{\mathrm{i}}

\newcommand{\spans}{\mathrm{span}}
\newcommand{\gap}{\mathrm{gap}}
\newcommand{\new}{\mathrm{new}}

\renewcommand{\Range}{\mathcal{R}}
\newcommand{\trace}{\mathrm{tr}}

\newcommand{\PP}{\mathcal{P}}
\newcommand{\QQ}{\mathcal{Q}}
\newcommand{\UU}{\mathcal{U}}
\newcommand{\VV}{\mathcal{V}}
\newcommand{\WW}{\mathcal{W}}
\newcommand{\XX}{\mathcal{X}}
\newcommand{\YY}{\mathcal{Y}}
\newcommand{\ZZ}{\mathcal{Z}}

\newcommand{\bsmallmatrix}[1]{\begin{bmatrix}\begin{smallmatrix}#1
\end{smallmatrix}\end{bmatrix}}

\title{A skew-symmetric Lanczos bidiagonalization method for
computing several largest eigenpairs of a large skew-symmetric matrix}

\author{
Jinzhi Huang\thanks{School of Mathematical Sciences, Soochow University, 215006 Suzhou, China
(\url{jzhuang21@suda.edu.cn}). The work of this author was supported in part
by the Youth Program of the Natural Science Foundation of Jiangsu Province (No. BK20220482). }
\and
Zhongxiao Jia\thanks{Corresponding author.
Department of Mathematical Sciences,
Tsinghua University, 100084 Beijing, China
(\url{jiazx@tsinghua.edu.cn}). The work of this author
was supported in part by the National Natural Science Foundation of China
(No. 12171273). }}

\begin{document}
\maketitle

\begin{abstract}
The spectral decomposition of a real skew-symmetric matrix $A$
can be mathematically transformed into a specific structured
singular value decomposition (SVD) of $A$.
Based on such equivalence, a skew-symmetric Lanczos
bidiagonalization (SSLBD) method is proposed for the specific
SVD problem that computes extreme singular values and the
corresponding singular vectors of $A$, from which the eigenpairs
of $A$ corresponding to the extreme conjugate eigenvalues in
magnitude are recovered pairwise in real arithmetic.
A number of convergence results on the method are established, and
accuracy estimates for approximate singular triplets are given.
In finite precision arithmetic, it is proven that the
semi-orthogonality of each set of basis vectors and the
semi-biorthogonality of two sets of basis vectors
suffice to compute the singular values accurately.
A commonly used efficient partial reorthogonalization strategy
is adapted to maintaining the needed semi-orthogonality and
semi-biorthogonality.
For a practical purpose, an implicitly restarted SSLBD
algorithm is developed with partial reorthogonalization.
Numerical experiments illustrate the effectiveness and
overall efficiency of the algorithm.
\end{abstract}



\pagestyle{myheadings}
\thispagestyle{plain}
\markboth{A skew-symmetric Lanczos bidiagonalization method}
{JINZHI HUANG AND ZHONGXIAO JIA}

\section{Introduction}\label{sec:1}
Let $A\in\mathbb{R}^{n\times n}$ be a large scale and possibly
sparse skew-symmetric matrix, i.e., $A^T=-A$, where the
superscript $T$ denotes the transpose of a  matrix or vector.
We consider the eigenvalue problem
\begin{equation}\label{eigA}
  Ax=\lambda x,
\end{equation}
where $\lambda\in\mathbb{C}$ is an eigenvalue of $A$ and
$x\in\mathbb{C}^{n}$ with the 2-norm
$\|x\|=1$ is an corresponding eigenvector.
It is well known that the nonzero eigenvalues $\lambda$ of $A$
are purely imaginary and come in conjugate pairs.
The real skew-symmetric eigenvalue problem arises in a variety
of applications, such as matrix function computations
\cite{cardoso2010exponentials,del2005computation},
the solution of linear quadratic optimal control problems
\cite{mehrmann1991autonomous,zhou1996robust},
model reductions \cite{mencik2015wave,yan1999approximate},
wave propagation solution for repetitive structures
\cite{haragus2008spectra,zhong1995direct},
crack following in anisotropic materials
\cite{apel2002structured,mehrmann2002polynomial},
and some others
\cite{mehrmann2012implicitly,penke2020high,wimmer2012algorithm}.

For small to medium sized skew-symmetric matrices,
several numerical algorithms have been developed
to compute their spectral decompositions in {\em real} arithmetic
\cite{fernando1998accurately,paardekooper1971eigenvalue,
penke2020high,ward1978eigensystem}.
For a large scale skew-symmetric $A$, however, it still lacks a
high-performance numerical algorithm that can make use
of its skew-symmetric structure to compute several eigenvalues
and/or the corresponding eigenvectors only in real arithmetic.
In this paper, we close this gap by proposing a Krylov subspace
type method to compute several extreme conjugate eigenvalues in
magnitude and the corresponding eigenvectors of $A$ that only
uses real arithmetic.

Our work originates from a basic fact that, as will be shown in
Section~\ref{sec2}, the spectral decomposition of the
skew-symmetric $A$ has a close relationship with its specifically
structured singular value decomposition (SVD), where each conjugate
pair of purely imaginary eigenvalues of $A$ corresponds to a double
singular value of $A$ and the mutually orthogonal real and
imaginary parts of associated eigenvectors are the corresponding
left and right singular vectors of $A$.
Therefore, the solution of eigenvalue problem of the skew-symmetric
$A$ is equivalent to the computation of its structured SVD, and the
desired eigenpairs can be recovered from the converged singular
triplets by exploiting this equivalence.

It has been well realized in \cite{ward1978eigensystem} that the
spectral decomposition of the skew-symmetric $A$ can be reduced
to the SVD of a certain bidiagonal matrix $B$ whose size is $n/2$
for $n$ even.
Exploiting such property, Ward and Gray \cite{ward1978eigensystem}
have designed an eigensolver for dense skew-symmetric matrices.
For $A$ large scale, however, the computation of $B$ using
orthogonal similarity transformations is prohibitive due to the
requirement of excessive storage and computations.
Nevertheless, with the help of Lanczos bidiagonalization (LBD)
\cite{paige1982}, whose complete bidiagonal reduction is originally
due to \cite{golub1965calculating}, we are able to compute a
sequence of leading principal matrices of $B$, based on which,
approximations to the extreme singular values of $A$ and/or
the corresponding singular vectors can be computed.
Along this line, a number of LBD based methods have been proposed
and intensively studied, and several practical explicitly and
implicitly restarted LBD type algorithms have been well developed \cite{baglama2005augmented,berry1992large,cullum1983lanczos,
	hochstenbach2004harmonic,jia2003implicitly,jia2010refined,
	larsen1998lanczos,larsen2001combining,stoll2012krylov}.

For the skew-symmetric matrix $A$, LBD owns some unique properties
in both mathematics and numerics.
However, direct applications of the aforementioned
algorithms do not make use of these special properties,
and will thus encounter some severe numerical difficulties.
We will fully exploit these properties to propose a skew-symmetric
LBD (SSLBD) method for computing several extreme singular triplets
of $A$, from which its extreme conjugate eigenvalues in magnitude
and the associate eigenvectors are recovered.
Attractively, our algorithm is performed in real arithmetic when
computing complex eigenpairs of $A$.
Particularly, one conjugate pair of approximations to two conjugate
pairwise eigenpairs can be recovered from {\em one} converged
approximate singular triplet.
We estimate the distance between the {\em subspace} generated by
the real and imaginary parts of a pair of desired conjugate
eigenvectors and the Krylov subspaces generated by SSLBD, and
prove how it tends to zero as the subspace dimension increases.
With these estimates, we establish a priori error bounds for the
approximate eigenvalues and approximate eigenspaces obtained by
the SSLBD method.
One of them is the error bound for the two dimensional approximate
eigenspace, which extends \cite[pp.103, Theorem 4.6]{saad2011numerical},
a classic error bound for the Ritz
vector in the real symmetric (complex Hermitian) case.
These results show how fast our method converges when computing
several extreme singular triplets.

As the subspace dimension increases, the SSLBD method will eventually
become impractical due to the excessive memory and computational cost.
For a practical purpose, it is generally necessary to restart the
method by selecting an increasingly better starting vector when the
maximum iterations allowed are attained, until the restarted algorithm
converges ultimately.
We will develop an implicitly restarted SSLBD algorithm.
Initially, implicit restart was proposed by Sorensen
\cite{sorensen1992implicit} for large eigenvalue problems,
and then it has been developed by Larsen~\cite{larsen2001combining}
and Jia and Niu \cite{jia2003implicitly,jia2010refined} for large
SVD computations.
Our implicit restart is based on them but specialized to the
skew-symmetric $A$.
We will show that two sets of left and right Lanczos vectors generated
by SSLBD are meanwhile biorthogonal for the skew-symmetric $A$.
However, in finite precision arithmetic, this numerical biorthogonality
loses gradually even if they themselves are kept numerically orthonormal.
As a matter of fact, we have found that
the numerical orthogonality of the computed left and right
Lanczos vectors in an available state-of-the-art Lanczos bidiagonalization
type algorithm does not suffice for the SVD problem of
the skew-symmetric $A$, and unfortunate ghosts
appear when the numerical biorthogonality loses severely; that is, a
singular value of $A$ may be recomputed a few times. Therefore,
certain numerical biorthogonality is crucial to make the
method work regularly.

As has been proved by Simon~\cite{simon1984analysis} for the symmetric
Lanczos method on the eigenvalue problem and extended by
Larsen \cite{larsen1998lanczos} to the LBD method for the SVD problem,
as far as the accurate computation of singular values is concerned,
the numerical semi-orthogonality
suffices for a general $A$, and the numerical orthogonality to the level
of machine precision is not necessary. Here semi-orthogonality
means that two vectors are numerically orthogonal to the level
of the square root of machine precision.
To maintain the numerical stability and make the SSLBD method behave like
it in exact arithmetic, we shall prove that the
semi-orthogonality of each set of Lanczos vectors
and the semi-biorthogonality of left and right Lanczos vectors
suffice for the accurate computation of singular values of
the skew-symmetric $A$. We will seek for an effective and
efficient reorthogonalization strategy for this purpose.
Specifically, besides the commonly used partial
reorthogonalization to maintain the desired semi-orthogonality
of each set of Lanczos vectors, we will
propose a partial reorthogonalization strategy to
maintain the numerical semi-biorthogonality of two sets of Lanczos vectors
in order to avoid the ghost phenomena and make
the algorithm converge regularly. We will introduce such kind of
reorthogonalization strategy into the implicitly restarted SSLBD
algorithm and design a simple scheme to recover the desired conjugate
eigenpairs of $A$ pairwise from the converged singular triplets.

The rest of this paper is organized as follows.
In Section~\ref{sec2}, we describe some properties of the spectral
decomposition of $A$, and show
its mathematical equivalence with the SVD of $A$.
Then we establish a close relationship between the SVD of $A$ and
that of an upper bidiagonal matrix whose order is half of that of $A$
when the order of $A$ is even.
In Section~\ref{sec3}, we present some properties of the SSLBD process,
and propose the SSLBD method for computing several
extreme singular values and the corresponding singular vectors
of $A$, from which conjugate approximate
eigenpairs of $A$ are pairwise reconstructed. In Section~\ref{converanal},
we establish convergence results on the SSLBD method.
In Section~\ref{sec:4}, we design a partial reorthogonalization scheme
for the SSLBD process, and develop an implicitly restarted SSLBD
algorithm to compute several extreme singular triplets of $A$.
Numerical experiments are reported in Section~\ref{sec:5}.
We conclude the paper in Section~\ref{sec:6}.

Throughout this paper, we denote by $\mi$ the imaginary unit,
by $X^H$ and $X^{\dag}$ the conjugate transpose and pseudoinverse
of $X$, by $\Range(X)$  the range space of $X$, and
by $\|X\|$ and $\|X\|_F$ the $2$- and Frobenius norms of $X$, respectively.
We denote by $\mathbb{R}^{k}$ and $\mathbb{C}^{k}$ the real and
complex spaces of dimension $k$,
by $I_{k}$ the identity matrix of order $k$, and by $\bm{0}_{k}$ and
$\bm{0}_{k,\ell}$ the zero matrices of orders $k\times k$ and $k\times \ell$,
respectively. The subscripts of identity and zero matrices are omitted when
they are clear from the context.

\section{Preliminaries} \label{sec2}
For the skew-symmetric matrix $A\in\mathbb{R}^{n\times n}$,
the following five properties are well known and easily justified:
Property (\romannumeral1): $A$ is diagonalizable by a
unitary similarity transformation;
Property (\romannumeral2): the eigenvalues $\lambda$ of $A$ are
either purely imaginary or zero;
Property (\romannumeral3): for an eigenpair $(\lambda,z)$ of $A$,
the conjugate $(\bar\lambda,\bar z)$ is also an eigenpair of $A$;
Property (\romannumeral4): the real and imaginary parts of the eigenvector $z$
corresponding to a purely imaginary eigenvalue of $A$ have the same length,
and are mutually orthogonal;
Property (\romannumeral5): $A$ of odd order must be singular, and has
at least one zero eigenvalue.
It is easily deduced that a nonsingular $A$ must have even
order, i.e., $n=2\ell$ for some integer $\ell$, and its eigenvalues
$\lambda$ are purely imaginary and
come in conjugate, i.e., plus-minus, pairs.

We formulate the following basic
results as a theorem, which establishes close structure relationships
between the spectral decomposition of $A$ and its SVD and
forms the basis of our method and algorithm to be proposed.
For completeness and our frequent use in this paper,
we will give a rigorous proof by exploiting the above five properties.

\begin{theoremmy}\label{thm1}
The spectral decomposition of the nonsingular skew-symmetric
$A\in\mathbb{R}^{n\times n}$ with $n=2\ell$ is
\begin{equation}\label{eigdecompA}
A=X\Lambda X^H
\end{equation}
with
\begin{equation}\label{defWL}
X=\begin{bmatrix}
\frac{1}{\sqrt 2}(U+\mi V)
&\frac{1}{\sqrt 2}(U-\mi V)\end{bmatrix}
\qquad\mbox{and}\qquad
\Lambda=\diag\{\mi \Sigma,-\mi\Sigma\},
\end{equation}
where $U\in\mathbb{R}^{n\times \ell}$ and
$V\in\mathbb{R}^{n\times \ell}$ are orthonormal and biorthogonal:
\begin{equation}\label{orthXYZ}
	U^TU=I,\qquad
	V^TV=I,\qquad
	U^TV=\bm{0},
\end{equation}
and $\Sigma=\diag\{\sigma_1,\dots,\sigma_{\ell}\}
\in\mathbb{R}^{\ell\times \ell}$
with $\sigma_1,\dots,\sigma_\ell>0$. The SVD of $A$ is
\begin{equation}\label{SVDA}
A=\begin{bmatrix}U &V\end{bmatrix}
\begin{bmatrix}\Sigma&\\&\Sigma\end{bmatrix}
\begin{bmatrix}V&-U\end{bmatrix}^T.
\end{equation}
\end{theoremmy}

\begin{proof}
By Properties (i) and (ii) and the assumption that $A$ is nonsingular,
we denote by $\Lambda_{+}\in\mathbb{C}^{\ell\times \ell}$
the diagonal matrix consisting
of all the eigenvalues of $A$ with positive imaginary parts, and
by $X_{+}\in\mathbb{C}^{n\times \ell}$ the corresponding orthonormal
eigenvector matrix. By
Property (iv), we can write $(\Lambda_{+},X_{+})$ as
\begin{equation}\label{defLW1}
(\Lambda_{+},X_{+})=\left(\mi\Sigma,\tfrac{1}{\sqrt 2}(U+\mi V)\right),
\end{equation}
where the columns of $U$ and $V$ have unit-length, and
$AX_+=X_+\Lambda_+$ and $X_{+}^HX_{+}=I$.
By these and Property (iii), the conjugate of $(\Lambda_+,X_+)$ is
\begin{equation}\label{defLW2}
(\Lambda_-,X_-)
=\left(-\mi\Sigma,\tfrac{1}{\sqrt 2}(U-\mi V)\right),
\end{equation}
and $AX_-=X_-\Lambda_-$ and
$X_{-}^HX_{-}=I$. As a result, $X$ and $\Lambda$ defined by
\eqref{defWL} are the eigenvector and eigenvalue matrices of $A$,
respectively, and
\begin{equation}\label{eigeq}
AX=X\Lambda.
\end{equation}

Premultiplying the two sides of $AX_+=X_+\Lambda_+$ by $X_+^T$ delivers
$X_+^TAX_+=X_+^TX_+\Lambda_+$,
which, by the skew-symmetry of $A$, shows that
\begin{equation*}
 X_+^TX_+\Lambda_++\Lambda_+X_+^TX_+=\bm{0}.
 \end{equation*}
Since $\Lambda_+=\mi\Sigma$ and $\Sigma$ is a positive diagonal matrix,
it follows from the above equation that
\begin{equation*}
X_+^TX_+=\bm{0}.
\end{equation*}
By definition \eqref{defWL} of $X$ and
the fact that $X_+$ and $X_-$ are column orthonormal, $X$ is unitary.
Postmultiplying \eqref{eigeq} by $X^H$ yields \eqref{eigdecompA}, from
which it follows that \eqref{SVDA} holds.

Inserting $X_+=\frac{1}{\sqrt2}(U+\mi V)$ into
$X_+^HX_+=I$ and $X_+^TX_+=\bm{0}$ and solving the
resulting equations for $U^TU, V^TV$
and $U^TV$, we obtain \eqref{orthXYZ}.
This shows that both $[U,V]$ and $[V,-U]$ are orthogonal.
Therefore, \eqref{SVDA} is the SVD of $A$.
\end{proof}

\begin{remark}\label{rem1}
For a singular skew-symmetric $A\in\mathbb{R}^{n\times n}$ with
$n=2\ell+\ell_0$ where $\ell_0$ is the multiplicity of zero eigenvalues of $A$,
we can still write the spectral decomposition of $A$
as \eqref{eigdecompA} with some slight rearrangements for the eigenvalue and
eigenvector matrices $\Lambda$ and $X$:
\begin{equation}\label{defXL2}
	X=\begin{bmatrix}
		\frac{1}{\sqrt 2}(U+\mi V)
		&\frac{1}{\sqrt 2}(U-\mi V)& X_0\end{bmatrix}
	\qquad\mbox{and}\qquad
	\Lambda=\diag\{\mi \Sigma,-\mi\Sigma,\bm{0}_{\ell_0}\},
\end{equation}
where $U,V$ and $\Sigma$ are as in Theorem~\ref{thm1} and the columns of $X_0$
form an orthonormal basis of the null space of $A$.
The proof is analogous to that of Theorem~\ref{thm1} and
thus omitted. In this case, relation \eqref{SVDA} gives the thin SVD of $A$.	
\end{remark}

For $U$ and $V$ in \eqref{SVDA}, write
$$
U=[u_1,u_2,\ldots,u_{\ell}]
\qquad\mbox{and}\qquad
V=[v_1,v_2,\ldots,v_{\ell}].
$$
On the basis of Theorem~\ref{thm1}, in the sequel, we denote by
\begin{equation}\label{eigpair}
(\lambda_{\pm j},x_{\pm j})=\left(\pm\mi\sigma_j,
\tfrac{1}{\sqrt2}(u_j\pm\mi v_j)\right),\qquad j=1,\dots,\ell.
\end{equation}
The SVD~\eqref{SVDA} indicates that $(\sigma_j,u_j,v_j)$ and
$(\sigma_j,v_j,-u_j)$ are two singular triplets of $A$ corresponding
to the multiple singular value $\sigma_j$.
Therefore, in order to obtain the conjugate eigenpairs
$(\lambda_{\pm j},x_{\pm j})$ of $A$, one can compute only the
singular triplet $(\sigma_j,u_j,v_j)$ or
$(\sigma_j,v_j,-u_j)$, and then recovers the desired eigenpairs
from the singular triplet by \eqref{eigpair}.

Next we present a bidiagonal decomposition of the nonsingular
skew-symmetric $A$ with order $n=2\ell$, which essentially appears in
\cite{ward1978eigensystem}. For completeness, we include a proof.

\begin{theoremmy}\label{thm2}
Assume that the skew-symmetric $A\in\mathbb{R}^{n\times n}$
with $n=2\ell$ is nonsingular.
Then there exist two orthonormal matrices
$P\in\mathbb{R}^{n\times \ell}$ and
$Q\in\mathbb{R}^{n\times \ell}$ satisfying
$P^TQ=\mathbf{0}$ and an upper bidiagonal matrix
$B\in\mathbb{R}^{\ell\times\ell}$  such that
\begin{equation}\label{bidiagA}
A=
\begin{bmatrix}P&Q\end{bmatrix}
\begin{bmatrix}B&\\&B^T\end{bmatrix}
\begin{bmatrix}Q&-P\end{bmatrix}^T.
\end{equation}
\end{theoremmy}

\begin{proof}
It is known from, e.g., \cite{ward1978eigensystem} that $A$ is
orthogonally similar to a skew-symmetric tridiagonal
matrix, that is, there exists an orthogonal
matrix $F\in\mathbb{R}^{n\times n}$
such that
\begin{equation}\label{defH}
  A^{\prime}=F^TAF=\begin{bmatrix}
0&-\alpha_1&&\\
\alpha_1&\ddots&\ddots&\\
&\ddots&\ddots&-\alpha_{n-1}\\
&&\alpha_{n-1}&0
\end{bmatrix}.
\end{equation}

Denote $\Pi_1=[e_{2,n},e_{4,n},\dots,e_{2\ell,n}]$ and
$\Pi_2=[e_{1,n},e_{3,n},\dots,e_{2\ell-1,n}]$
with $e_{j,n}$ being the $j$th column of $I_n$, $j=1,\dots,n$.
Then it is easy to verify that
\begin{eqnarray}
	A^{\prime\prime}&=&\ \begin{bmatrix}\Pi_1&\Pi_2\end{bmatrix}^TA^{\prime}
	\begin{bmatrix}\Pi_2&\Pi_1\end{bmatrix}
	=\begin{bmatrix}B&\\&-B^T\end{bmatrix},\nonumber \\
	B\ &=&\ \Pi_1^TA^{\prime}\Pi_2=\begin{bmatrix}
		\alpha_1&-\alpha_2&&&\\[-0.1em]
		&\alpha_3&-\alpha_4&&\\[-0.1em]
		&&\ddots&\ddots&\\[-0.1em]
		&&&\ddots&-\alpha_{2\ell-2}\\[-0.1em]
		&&&&\alpha_{2\ell-1}
	\end{bmatrix}.\label{defB}
\end{eqnarray}
Combining these two relations with \eqref{defH}, we obtain
\begin{eqnarray}
A&=&FA^{\prime}F^T=F\begin{bmatrix}\Pi_1&\Pi_2\end{bmatrix}
A^{\prime\prime}
\begin{bmatrix}\Pi_2&\Pi_1\end{bmatrix}^TF^T \nonumber\\
&=&\begin{bmatrix}F\Pi_1&F\Pi_2\end{bmatrix}A^{\prime\prime}
\begin{bmatrix}F\Pi_2&F\Pi_1\end{bmatrix}^T \nonumber \\
&=&\begin{bmatrix}P&Q\end{bmatrix}
\begin{bmatrix}B&\\&-B^T\end{bmatrix}
\begin{bmatrix}Q&P\end{bmatrix}^T,\label{skewbidag}
\end{eqnarray}
which proves \eqref{bidiagA}, where we have denoted
$P=F\Pi_1$ and $Q=F\Pi_2$.
Note that $F$ is orthogonal and $\Pi_1$, $\Pi_2$ are not only
orthonormal but also biorthogonal:
$\Pi_1^T\Pi_1=\Pi_2^T\Pi_2=I$ and $\Pi_1^T\Pi_2=\mathbf{0}$.
Therefore, $P^TP=Q^TQ=I$ and $P^TQ=\bm{0}$.
\end{proof}

\begin{remark}\label{rem2}
For a singular skew-symmetric $A\in\mathbb{R}^{n\times n}$ with
$n=2\ell+\ell_0$ where $\ell_0$ is the multiplicity of zero
eigenvalues of $A$, the lower diagonal elements
$\alpha_i,i=2\ell,2\ell+1,\dots,n-1$ of
$A^{\prime}$ in \eqref{defH} are zeros.
Taking $\Pi_1$ and $\Pi_2$ as in \eqref{defB} and $\Pi_3=[e_{2\ell+1,n},e_{2\ell+2,n},\dots,e_{n,n}]$, we obtain
\begin{equation}\label{defB2}
	A^{\prime\prime}=\ \begin{bmatrix}\Pi_1&\Pi_2&\Pi_3\end{bmatrix}^TA^{\prime}
	\begin{bmatrix}\Pi_2&\Pi_1&\Pi_3\end{bmatrix}
	=\diag\{B,-B^T, \bm{0}_{\ell_0}\}
\end{equation}
with $B$ defined by \eqref{defB}.
Using the same derivations as in the proof of Theorem~\ref{thm2},
we have
\begin{equation*}
	A=\begin{bmatrix}P&Q&X_0\end{bmatrix}
	\begin{bmatrix}B\!\!&&\\&\!\!B^T\!\!&\\&&\!\!\bm{0}_{\ell_0}\end{bmatrix}
	\begin{bmatrix}Q&-P&X_0\end{bmatrix}^T,	
\end{equation*}
where $P$ and $Q$ are as in \eqref{skewbidag} and $X_0=F\Pi_3$
is the orthonormal basis matrix of the null space of $A$.
This decomposition is a precise extension of \eqref{bidiagA}
to $A$ with any order.
\end{remark}

Combining Remark~\ref{rem2} with Remark~\ref{rem1} after
Theorem~\ref{thm1}, for ease of presentation, in the sequel
we will always assume that $A$ has even order and is nonsingular.
However, our method and algorithm to be proposed apply to a singular
skew-symmetric $A$ with only some slight modifications.

Decomposition \eqref{bidiagA} is a bidiagonal decomposition of $A$
that reduces $A$ to $\diag\{B,B^T\}$ using the structured orthogonal
matrices $\begin{bmatrix}P&Q\end{bmatrix}^T$ and
$\begin{bmatrix}Q&-P\end{bmatrix}$ from the left and right,
respectively. The singular values of $A$ are the union of
those of $B$ and $B^T$ but the SVD of $B$ alone delivers
the whole SVD of $A$:
Let $B=U_B\Sigma_B V_B^T$ be the SVD of $B$.
Then Theorem~\ref{thm2} indicates that $(\Sigma,U,V)=(\Sigma_B,PU_B,QV_B)$
is an $\ell$-dimensional partial SVD of $A$ with $U$ and $V$
being orthonormal and $U^TV=\mathbf{0}$.
By Theorem~\ref{thm1}, the spectral decomposition
\eqref{eigdecompA} of $A$ can be restored directly from $(\Sigma,U,V)$.
This is exactly the working mechanism of the
eigensolver proposed in \cite{ward1978eigensystem} for small to
medium sized skew-symmetric matrices.

For $A$ large, using orthogonal transformations to construct
$P$, $Q$ and $B$ in \eqref{bidiagA} and computing the SVD of
$B$ are unaffordable.
Fortunately, based on decomposition \eqref{bidiagA},
we can perform the LBD process on the skew-symmetric $A$, i.e.,
the SSLBD process, to successively generate
the columns of $P$ and $Q$ and the leading principal matrices of $B$, so
that the Rayleigh--Ritz projection comes
into the computation of a partial SVD of $A$.

\section{The SSLBD method}\label{sec3}
For ease of presentation, from now on,
we always assume that the eigenvalues $\lambda_{\pm j}
=\pm\mi\sigma_j$ of $A$ in \eqref{eigpair} are simple, and that
$\sigma_1,\dots,\sigma_\ell$ are labeled in decreasing order:
\begin{equation}\label{labeleig}
  \sigma_1>\sigma_2>\dots>\sigma_\ell>0.
\end{equation}
Our goal is to compute the $k$ pairs of extreme, e.g., largest conjugate
eigenvalues $\lambda_{\pm 1},\dots,\lambda_{\pm k}$
and/or the corresponding eigenvectors $x_{\pm 1},\dots,x_{\pm k}$ defined
by \eqref{eigpair}.
By Theorem~\ref{thm1}, this amounts to computing the following partial SVD of
$A$:
$$
(\Sigma_k,U_k,V_k)=\left(\mathrm{diag}\{\sigma_1,\dots,\sigma_k\},
[u_1,\dots,u_k],[v_1,\dots,v_k]\right).
$$

\subsection{The $m$-step SSLBD process}\label{subsec:1}
Algorithm~\ref{alg1} sketches the $m$-step SSLBD process on the
skew-symmetric $A$, which is a variant of the lower LBD
process proposed by Paige and Saunders \cite{paige1982}.

\begin{algorithm}[htbp]
\caption{The $m$-step SSLBD process.}
\begin{algorithmic}[1]\label{alg1}
\STATE{Initialization:\ Set $\gamma_0=0$ and $p_{0}=\bm{0}$,
and choose $q_1\in\mathbb{R}^{n}$ with $\|q_1\|=1$.}
\FOR{$j=1,\dots,m<\ell$}
\STATE{Compute $s_j=Aq_j-\gamma_{j-1}p_{j-1}$ and
	$\beta_j=\|s_j\|$.} \label{skwbld1}

\STATE{\textbf{if} $\beta_j=0$ \textbf{then} break;
\textbf{else} calculate $p_j=\frac{1}{\beta_j}s_j$.}  \label{skwbld2}

\STATE{Compute $t_j=-Ap_j-\beta_jq_j$ and
	$\gamma_{j}=\|t_j\|$.} \label{skwbld3}

\STATE{\textbf{if} $\gamma_j=0$ \textbf{then} break;
\textbf{else} calculate $q_{j+1}=\frac{1}{\gamma_{j}}t_j$.} \label{skwbld4}
\ENDFOR
\end{algorithmic}
\end{algorithm}

Assume that the $m$-step SSLBD process does not
break down for $m<\ell$, and note that $A^TA=-A^2$ and $AA^T=-A^2$.
Then the process computes the orthonormal base $\{p_j\}_{j=1}^m$
and $\{q_j\}_{j=1}^{m+1}$ of the Krylov subspaces
\begin{equation}\label{defUV}
\UU_m=\mathcal{K}(A^2,m,p_1)
\qquad\mbox{and}\qquad
\VV_{m+1}=\mathcal{K}(A^2,m+1,q_1)
\end{equation}
generated by $A^2$ and the starting vectors $p_1$ and $q_1$ with
$p_1=Aq_1/\|Aq_1\|$, respectively. Denote
$P_j=[p_1,\dots,p_j]$ for $j=1,\dots,m$ and
$Q_j=[q_1,\dots,q_j]$ for $j=1,\dots,m+1$, whose columns are
called left and right Lanczos vectors,
respectively.
Then the $m$-step SSLBD process can be written in the matrix form:
\begin{equation}\label{LBDmat}
	\left\{\begin{aligned}
		&AQ_m=P_mB_m,\\[0.5em]
		&AP_m=-Q_{m} B_{m}^T-\gamma_mq_{m+1}e_{m,m}^T,
	\end{aligned}\right.
\qquad\mbox{with}\qquad
B_m=\begin{bmatrix}
	\beta_1&\gamma_1&&\\&\ddots&\ddots&\\&&\ddots&\gamma_{m-1}\\&&&\beta_{m}
	\end{bmatrix},
\end{equation}
where $e_{m,m}$ is the $m$th column of $I_m$.
It is well known that the singular values of $B_m$ are simple
whenever the $m$-step SSLBD process does not break down.
For later use, denote $\widehat B_m=[B_m,\gamma_me_{m,m}]$.
Then we can write the second relation in \eqref{LBDmat}
as $AP_m=-Q_{m+1}\widehat B_m^T$.

\begin{theoremmy}\label{thm3}
Let $P_m$ and $Q_{m+1}$ be generated by Algorithm~\ref{alg1}
without breakdown for $m<\ell$. Then $P_{m}$ and $Q_{m+1}$
are biorthogonal:
\begin{equation}\label{orthUV}
Q_{m+1}^TP_m=\bm{0}.
\end{equation}
\end{theoremmy}

\begin{proof}
We use induction on $m$.
Since $\beta_1p_1=Aq_1$ and $\gamma_1q_2=-Ap_1-\beta_1q_1$,
by the skew-symmetry of $A$ and
$\beta_1>0$, $\gamma_1>0$, we have
$$
\beta_1 q_1^T p_1=q_1^TAq_1=0 \qquad\mbox{and}\qquad
\gamma_1q_2^Tp_1=p_1^TAp_1-\beta_1 q_1^Tp_1=0,
$$
which means that $Q_2^TP_1=\bm{0}$ and thus
proves \eqref{orthUV} for $m=1$.

Suppose that \eqref{orthUV} holds for $m=1,2,\dots, j-1$,
i.e., $Q_{j}^TP_{j-1}=\bm{0}$.
For $m=j$, partition
\begin{equation}\label{orthUV0}
  Q_{j+1}^TP_j=\begin{bmatrix}Q_j^T\\q_{j+1}^T\end{bmatrix}P_j
  =\begin{bmatrix}Q_j^TP_j\\q_{j+1}^TP_j\end{bmatrix}.
\end{equation}
It follows from \eqref{LBDmat} and the
inductive hypothesis $Q_j^TP_{j-1}=\bm{0}$ that
$$
Q_j^TAQ_{j-1}=Q_j^TP_{j-1}B_{j-1}=\bm{0},
$$
which means
that $Q_{j-1}^TAQ_{j-1}=\bm{0}$ and $q_j^TAQ_{j-1}=\bm{0}$.
Since $\beta_i>0$, $i=1,\dots,j$, $B_j$ is nonsingular,
from \eqref{LBDmat} and the above relations we have
\begin{equation}\label{orthUV1}
  Q_j^TP_j=Q_j^TAQ_jB_j^{-1}
  =\begin{bmatrix}
  Q_{j-1}^TAQ_{j-1}&Q_{j-1}^TAq_j\\q_j^TAQ_{j-1}&q_j^TAq_j
  \end{bmatrix}B_{j}^{-1}=\bm{0}
\end{equation}
by noting that $q_j^TAq_j=0$.
Particularly, relation~\eqref{orthUV1} shows that
$q_j^TP_j=\bm{0}$ and $p_j^TQ_j=\bm{0}$.
Making use of them and $AP_{j-1}=-Q_{j}\widehat B_{j-1}^T$,
and noticing that $\gamma_j>0$, we obtain
\begin{eqnarray}
  \quad q_{j+1}^TP_j
  &=&\tfrac{1}{\gamma_j}(A^Tp_j-\beta_jq_j)^TP_j
  =\tfrac{1}{\gamma_j}p_j^TAP_j \nonumber\\
  &=&\tfrac{1}{\gamma_j}
  \begin{bmatrix}p_j^TAP_{j-1}&p_j^TAp_j\end{bmatrix}
  =\tfrac{1}{\gamma_j}
  \begin{bmatrix}-p_j^TQ_j\widehat B_{j-1}^T&\bm{0}\end{bmatrix}=\bm{0}.
      \label{orthUV2}
\end{eqnarray}
Applying \eqref{orthUV1} and \eqref{orthUV2} to
\eqref{orthUV0} yields \eqref{orthUV} for $m={j}$.
By induction, we have proved that relation \eqref{orthUV}
holds for all $m<\ell$.
\end{proof}

Theorem~\ref{thm3} indicates that any vectors from
$\UU_m=\Range(P_m)$ and $\VV_m=\Range(Q_m)$, which are
called left and right subspaces, are biorthogonal.
Therefore, {\em any} left and right approximate singular
vectors extracted from $\UU_m$ and $\VV_m$ are biorthogonal,
a desired property as the left and right singular vectors
of $A$ are so, too.

The SSLBD process must break down no later than $\ell$ steps
since $A$ has $\ell=n/2$ distinct singular values;
see \cite{jia2020} for a detailed analysis on the multiple
singular value case.
Once breakdown occurs for some $m<\ell$, all the singular
values of $B_m$ are exact ones of $A$, and $m$ exact
singular triplets of $A$ have been found, as we shall
see in Section~\ref{converanal}.
Furthermore, by the assumption that $A$ is nonsingular,
breakdown can occur only for $\gamma_m=0$;
for if $\beta_m=0$, then $B_m$ and thus $A$ would have
a zero singular value.

\subsection{The SSLBD method}\label{subsec:2}
With the left and right searching subspaces $\UU_m$ and $\VV_m$,
we use the standard Rayleigh--Ritz projection, i.e.,
the standard extraction approach
\cite{baglama2005augmented,berry1992large,cullum1983lanczos,
	jia2003implicitly,larsen1998lanczos,larsen2001combining},
to compute the approximate singular triplets
$\left(\theta_j,\tilde u_j,\tilde v_j\right)$
of $A$ with the unit-length vectors $\tilde u_j\in\UU_m$
and $\tilde v_j\in\VV_m$ that satisfy the requirement
\begin{equation}\label{standard}
	\left\{\begin{aligned}
		&A\tilde v_j-\theta_j\tilde u_j\perp\UU_m,\\
		&A\tilde u_j+\theta_j\tilde v_j\perp\VV_m,
	\end{aligned}\right.
	\qquad\qquad j=1,\dots,m.
\end{equation}
Set $\tilde u_j=P_mc_j$ and $\tilde v_j=Q_md_j$ for $j=1,\dots,m$.
Then \eqref{LBDmat} indicates that \eqref{standard} becomes
\begin{equation}\label{Ritz}
	B_md_j=\theta_jc_j,\qquad
	B_m^Tc_j=\theta_jd_j,\qquad
	j=1,\dots,m;
\end{equation}
that is, $(\theta_j,c_j,d_j)$, $j=1,\dots,m$ are the singular
triplets of $B_m$.
Therefore, the standard extraction amounts to computing the SVD
of $B_m$ with the singular values ordered as
$\theta_1>\theta_2>\dots>\theta_m$ and takes
$(\theta_j,\tilde u_j,\tilde v_j),\ j=1,\dots,m$ as
approximations to some singular triplets of $A$.
The resulting method is called the SSLBD
method, and the $(\theta_j,\tilde u_j, \tilde v_j)$ are called
the Ritz approximations of $A$ with respect to the left and
right searching subspaces $\UU_m$ and $\VV_m$; particularly,
the $\theta_j$ are the Ritz values, and the
$\tilde u_j$ and $\tilde v_j$ are the left and right
Ritz vectors, respectively.

Since $Q_m^TA^TAQ_m=B_m^TB_m$, by the Cauchy interlace theorem of
eigenvalues (cf. \cite[Theorem 10.1.1, pp.203]{parlett1998symmetric}),
for $j$ fixed, as $m$ increases, $\theta_j$ and $\theta_{m-j+1}$
monotonically converge to $\sigma_j$ and $\sigma_{\ell-j+1}$
from below and above, respectively.
Therefore, we can take
$(\theta_j,\tilde u_j,\tilde v_j),\ j=1,\dots,k$ as
approximations to the $k$ largest singular triplets
$(\sigma_j,u_j,v_j)$, $j=1,\dots,k$ of $A$ with $k\ll m$.
Likewise, we may use $(\theta_{m-j+1},\tilde u_{m-j+1},
\tilde v_{m-j+1})$, $j=1,\dots,k$ to approximate the $k$ smallest
singular triplets $(\sigma_{\ell-j+1},u_{\ell-j+1},v_{\ell-j+1})$,
$j=1,\dots,k$ of $A$.

\section{A convergence analysis}\label{converanal}
For later use, for each $j=1,\dots,\ell$, we denote
\begin{equation}\label{defWL2}
	\Lambda_j=\begin{bmatrix}&\sigma_j\\-\sigma_j&\end{bmatrix}
	\qquad\mbox{and}\qquad
X_j=\begin{bmatrix}u_j&v_j\end{bmatrix}.
\end{equation}
Then $AX_j=X_j\Lambda_j$, and $\XX_j:=\Range(X_j)=\spans\{x_{\pm j}\}$
is the two dimensional eigenspace of $A$ associated
with the conjugate eigenvalues $\lambda_{\pm j}$, $j=1,\dots,\ell$.
We call the pair $(\Lambda_j,X_j)$ a block eigenpair of $A$,
$j=1,\dots,\ell$. From \eqref{orthXYZ},
$\widehat X=[X_1,\dots,X_\ell]$ is orthogonal, and
\eqref{eigdecompA} can be written as
\begin{equation}\label{eigdecA2}
  A=\widehat X\widehat\Lambda\widehat X^T
\qquad\mbox{with}\qquad
\widehat\Lambda=\diag\{\Lambda_1,\dots,\Lambda_\ell\}.
\end{equation}

To make a convergence analysis on the SSLBD method,
we review some preliminaries. Let the columns of $Z$, $W$
and $W_{\perp}$ form orthonormal base of $\mathcal{Z}$,
$\mathcal{W}$ and the orthogonal complement of $\mathcal{W}$,
respectively. Denote by $\angle(\mathcal{W},\mathcal{Z})$
the canonical angles between $\mathcal{W}$ and $\mathcal{Z}$
\cite[pp.329--330]{golub2012matrix}.
The 2-norm distance between $\mathcal{Z}$ and $\mathcal{W}$
is defined by
$$
\|\sin\angle(\mathcal{W},\mathcal{Z})\|=\|W_{\perp}^HZ\|=\|(I-WW^H)Z\|,
$$
which is the sine of the largest canonical angle between
$\mathcal{W}$ and $\mathcal{Z}$ (cf. Jia and Stewart~\cite{jia2001}).
It is worthwhile to point out that if the dimensions of
$\mathcal{W}$ and $\mathcal{Z}$ are not equal then this
measure is not symmetric in its arguments.
In fact, if the dimension of $\mathcal{W}$ is greater
than that of $\mathcal{Z}$, then $\|\sin\angle(\mathcal{Z},
\mathcal{W})\|=1$, although generally
$\|\sin\angle(\mathcal{W},\mathcal{Z})\|<1$.
In this paper, we will use the $F$-norm distance
\begin{equation}\label{sinf}
\|\sin\angle(\mathcal{W},\mathcal{Z})\|_F=\|W_{\perp}^HZ\|_F,
\end{equation}
which equals the square root of
the squares sum of sines of all the canonical angles
between the two subspaces.
Correspondingly, we define
$\|\tan\angle(\mathcal{W},\mathcal{Z})\|_F$ to be
the square root of the squares sum of the tangents
of all the canonical angles between the subspaces.
These tangents are the generalized singular values of
the matrix pair $\{W_{\perp}^HZ, W^HZ\}$.

Notice that for $\mathcal{W}$ and $\mathcal{Z}$ with
{\em equal dimensions}, if $W^HZ$ is nonsingular then
the tangents of canonical angles are the singular
values of $W_{\perp}^HZ(W^HZ)^{-1}$, so that
\begin{equation}\label{tanf}
\|\tan\angle(\mathcal{W},\mathcal{Z})\|_F
=\|W_{\perp}^HZ(W^HZ)^{-1}\|_F=\|(I-WW^H)Z(W^HZ)^{-1}\|_F,
\end{equation}
which holds when the $F$-norm is replaced by the 2-norm.
We remark that the inverse in the above cannot be
replaced by the pseudoinverse $\dagger$ when $W^HZ$ is singular.
In fact, by definition, $\|\tan\angle(\mathcal{W},\mathcal{Z})\|_F$
is infinite when $W^HZ$ is singular, and the
generalized singular values of $\{W_{\perp}^HZ, W^HZ\}$
are not the singular values of $W_{\perp}^HZ(W^HZ)^{\dag}$
in this case.

Note that the real and imaginary parts of the eigenvectors
$x_{\pm j}$ of $A$ can be interchanged since $\mi x_{\pm j}$
are the eigenvectors of $A$ and its left and right singular
vectors $u_j$ and $v_j$ are the right and left ones, too
(cf. \eqref{SVDA}).
As we have pointed out, any pair of approximate left and
right singular vectors extracted from the biorthogonal
$\UU_m$ and $\VV_m$ are mutually orthogonal, using which
we construct the real and imaginary parts of an approximate
eigenvector of $A$.
As a consequence, in the SVD context of the skew-symmetric $A$,
because of these properties, when analyzing the convergence
of the SSLBD method, we consider the orthogonal direct sum
$\UU_m\oplus\VV_m$ as a whole, and estimate the distance
between a desired two dimensional eigenspace $\XX_j$ of
$A$ and the $2m$-dimensional subspace $\UU_m\oplus\VV_m$
for $j$ small. We remind that their dimensions are unequal
for $m>1$.
Using these distances and their estimates, we can establish
a priori error bounds for approximate eigenvalues and
approximate eigenspaces obtained by the SSLBD method,
showing how fast they converge when computing largest
eigenvalues in magnitudes and the associated eigenvectors.

In terms of the definition of
$\|\tan\angle(\mathcal{W},\mathcal{Z})\|_F$,
we present the following estimate for
$\|\tan\angle(\UU_m\oplus\VV_m,\mathcal{X}_j)\|_F$.

\begin{theoremmy}\label{thm4}
Let $\UU_m$ and $\VV_m$ be defined by \eqref{defUV},
and suppose that $X_j^HY$ with the initial
$Y=[p_1,q_1]$ is nonsingular.
Then the following estimate holds for any integer $1\leq j<m$:
\begin{equation}\label{accUV}
\|\tan\angle(\UU_m\oplus\VV_m,\mathcal{X}_j)\|_F\leq
\frac{\eta_j}{\chi_{m-j}(\xi_j)}\|\tan\angle(\mathcal{Y},\mathcal{X}_j)\|_F,
\end{equation}
where $\chi_{i}(\cdot)$ is the degree $i$ Chebyshev
polynomial of the first kind and
\begin{equation}\label{gammarho}
  \xi_j=1+2\cdot\frac{\sigma_j^2-\sigma_{j+1}^2}
  {\sigma^2_{j+1}-\sigma^2_{\ell}}
  \qquad\mbox{and}\qquad
  \eta_j=\left\{\begin{aligned}
  &1,\qquad \quad\qquad\qquad\mbox{if} \quad  j=1,\\
  &\prod\limits_{i=1}^{j-1}\frac{\sigma_i^2-\sigma_\ell^2}
  {\sigma_i^2-\sigma_j^2}, \hspace{2.4em}\mbox{if}\quad j>1.
  \end{aligned}\right.
\end{equation}
\end{theoremmy}

\begin{proof}
For a fixed $1\leq j< m$, denote by
$\widehat \Lambda_j=\diag\{\Lambda_1,\dots,
\Lambda_{j-1},\Lambda_{j+1},\dots,\Lambda_\ell\}$
and $\widehat X_j=\![X_1,\dots,X_{j-1},X_{j+1},\dots,X_\ell]$
which delete $\Lambda_j$ and $X_j$ from $\widehat X$
and $\widehat\Lambda$ in \eqref{eigdecA2}, respectively.
Then relation \eqref{eigdecA2} can be written as
\begin{equation}\label{eigdecA3}
  A=\widehat X_j\widehat \Lambda_j\widehat X_j^T+X_j\Lambda_jX_j^T.
\end{equation}

Notice that $Y=[p_1,q_1]$ is orthonormal by Theorem~\ref{thm3},
and $\mathcal{Y}=\Range(Y)=\UU_1\oplus\VV_1$.
Since $\widehat X$ in \eqref{eigdecA2} is orthogonal,
there exists an orthonormal matrix $G=[G_1^T,\dots,G_\ell^T]^T$ with
$G_1,\dots,G_\ell\in\mathbb{R}^{2\times 2}$ such that
\begin{equation}\label{expv1}
  Y=\widehat XG=\widehat X_j\widehat G_j+X_jG_j
  \qquad\mbox{with}\qquad
  \widehat G_j=[G_1,\dots,G_{j-1},G_{j+1},\dots,G_\ell].
\end{equation}
For an arbitrary polynomial $\rho(\cdot)$ in $\mathcal{P}^{m-1}$,
the set of polynomials of degree no more than $m-1$ that satisfies
$\rho(-\sigma_j^2)\neq0$, write $Y_{\rho}=\rho(A^2)Y$
and $\YY_{\rho}=\Range(Y_{\rho})$.
Then
\begin{equation}\label{yrho}
Y_{\rho}=\widehat X_j\rho(\widehat \Lambda_j^2)\widehat G_j
+X_j\rho(\Lambda_j^2)G_j.
\end{equation}
From \eqref{defWL2}, it is known that
$\Lambda_i^2=-\sigma_i^2I_2$, $i=1,\dots,\ell$.

Since $G_j=X_j^HY$ is nonsingular, $\rho(-\sigma_j^2)\neq0$,
and the columns of $Y_{\rho}(Y_{\rho}^TY_{\rho})^{-1/2}$
form an orthonormal basis of $\mathcal{Y}_{\rho}$,
by definition~\eqref{tanf} and relation~\eqref{yrho}, we obtain
\begin{eqnarray}
\|\tan\angle(\mathcal{Y}_{\rho},\mathcal{X}_j)\|_F
&=&\|(I-X_jX_j^T)Y_{\rho}(Y_{\rho}^TY_{\rho})^{-1/2}(X_j^T
Y_{\rho}(Y_{\rho}^TY_{\rho})^{-1/2})^{-1}\|_F  \nonumber\\
&=&\|(I-X_jX_j^T)Y_{\rho}(X_j^TY_{\rho})^{-1}\|_F
=\|\rho(\widehat \Lambda_j^2)\widehat G_j (\rho(\Lambda_j^2)G_j)^{-1}\|_F
\nonumber\\
&\leq&\frac{\|\rho(\widehat \Lambda_j^2)\|}{|\rho(-\sigma_j^2)|}
\|\widehat G_jG_j^{-1}\|_F
=\frac{\max_{i\neq j}|\rho(-\sigma_i^2)|}{|\rho(-\sigma_j^2)|}
\|\widehat G_jG_j^{-1}\|_F \nonumber\\
&=&\frac{\max_{i\neq j}|\rho(-\sigma_i^2)|}{|\rho(-\sigma_j^2)|}
\|(I-X_jX_j^T)Y(X_j^TY)^{-1}\|_F
\qquad\qquad\ \ \mbox{by \eqref{expv1}}\hspace{-4em}
\nonumber \\
&=&\frac{\max_{i\neq j}|\rho(-\sigma_i^2)|}{|\rho(-\sigma_j^2)|}
\|\tan\angle(\mathcal{Y},\mathcal{X}_j)\|_F
\label{tanpsi}.
\end{eqnarray}

Note that $\YY_{\rho}\subset \UU_m\oplus\VV_m$ for any
$\rho\in\PP^{m-1}$.
By definition and \eqref{tanpsi}, it holds that
\begin{eqnarray*}
\|\tan\angle(\UU_m\oplus\VV_m,\mathcal{X}_j)\|_F&\leq&
\min_{\rho\in\mathcal{P}^{m-1}}
\|\tan\angle(\mathcal{Y}_{\rho},\mathcal{X}_j)\|_F \\
&\leq&
\min\limits_{\rho\in\mathcal{P}^{m-1}}
\frac{\max_{i\neq j}|\rho(-\sigma_i^2)|}{|\rho(-\sigma_j^2)|}
\|\tan\angle(\mathcal{Y},\mathcal{X}_j)\|_F \\
&=&
\min\limits_{\rho\in\mathcal{P}^{m-1}, \rho(-\sigma_j^2)=1}
\max_{i\neq j}|\rho(-\sigma_i^2)|
\|\tan\angle(\mathcal{Y},\mathcal{X}_j)\|_F\\
&\leq&
\frac{\eta_j}{\chi_{m-j}(\xi_j)}
\|\tan\angle(\mathcal{Y},\mathcal{X}_j)\|_F,
\end{eqnarray*}
where the last inequality comes from the proof of Theorem~1
in \cite{saad1980rates}, $\chi_{m-j}(\cdot)$ is the degree $m-j$
Chebyshev polynomial of the first kind, and $\xi_j$ and $\eta_j$
are defined by \eqref{gammarho}.
\end{proof}

Theorem~\ref{thm4} establishes accuracy estimates for
$\mathcal{X}_j$ approximating the searching subspaces $\UU_m\oplus\VV_m$.
But one must be aware that they are only significant for $j\ll m$.
The scalars $\eta_j\geq 1$ and $\xi_j>1$ defined by \eqref{gammarho}
are constants depending only on the eigenvalue or singular value
distribution of $A$.
For a fixed integer $j\ll m$, as long as the initial searching subspace
$\mathcal{Y}$ contains some information on $\XX_j$,
that is, $\|\tan\angle(\mathcal{Y},\XX_j)\|_F$ is finite in \eqref{accUV},
then the larger $m$ is,
the smaller $\frac{\eta_j}{\chi_{m-j}(\xi_j)}$ is and the
closer $\XX_j$ is to $\UU_m\oplus\VV_m$.
Moreover, the better the singular value $\sigma_j$ is separated
from the other singular values $\sigma_i\neq\sigma_j$ of $A$,
the smaller $\eta_j$ and the larger $\xi_j$ are, meaning that $\XX_j$
approaches $\UU_m\oplus\VV_m$ more quickly as $m$ increases.
Generally, $\|\tan\angle(\UU_m\oplus\VV_m,\XX_j)\|_F$
decays faster for $j$ smaller.
Two extreme cases are that $\|\tan\angle(\mathcal{Y},\XX_j)\|_F=0$
and $\|\tan\angle(\mathcal{Y},\XX_j)\|_F=+\infty$.
In the first case, Algorithm~\ref{alg1} breaks down at step $m=1$,
and we already have the exact eigenspace $\XX_j=\mathcal{Y}$.
The second case indicates that the initial $\mathcal{Y}$ is deficient
in $\XX_j$, such that $\UU_m\oplus\VV_m$ does not contain any
information on $\XX_j$; consequently, one cannot
find any meaningful approximations to $u_j$ and $v_j$ from $\UU_m$ and
$\VV_m$ for any $m<\ell$.

Theorem~\ref{thm4} shows that the bounds tend to zero for $j$
small as $m$ increases. Using a
similar proof, we can establish the following analogue of \eqref{accUV}
for $j$ small:
\begin{equation}\label{accUVs}
\|\tan\angle(\UU_m\oplus\VV_m,\mathcal{X}_{\ell-j+1})\|_F\leq
\frac{\hat{\eta}_j}{\chi_{m-j}(\hat{\xi}_j)}\|\tan\angle(\mathcal{Y},
\mathcal{X}_{\ell-j+1})\|_F,
\end{equation}
where
$$
\hat{\xi}_j=1+2\cdot\frac{\sigma_{\ell-j+1}^2-
	\sigma_{\ell-j}^2}{\sigma^2_{\ell-j}-\sigma^2_1}
\qquad\mbox{and}\qquad
\hat{\eta}_j=\left\{\begin{aligned}
	&1,\qquad \quad\qquad\qquad\ \ \ \mbox{if} \quad  j=1,\\
	&\prod\limits_{i=1}^{j-1}\frac{\sigma_{\ell-i}^2-\sigma_1^2}
	{\sigma_{\ell-i}^2-\sigma_{\ell-j}^2},
	\hspace{1.85em}\mbox{if}\quad j>1.
\end{aligned}\right.	
$$
Similar arguments show that
$\|\tan\angle(\UU_m\oplus\VV_m,\mathcal{X}_{\ell-j+1})\|_F$ generally
tends to zero faster for $j$ smaller as $m$ increases. It indicates
that $\UU_m\oplus\VV_m$ also favors the eigenvectors of $A$ corresponding
to several smallest eigenvalues in magnitude if the
smallest singular values $\sigma_{\ell-j+1}$ are not clustered.
In the sequel, for brevity, we only discuss the computation of the largest
conjugate eigenvalues in magnitude and the associated eigenvectors of $A$.

In term of $\|\tan\angle(\UU_m\oplus\VV_m,\mathcal{X}_j)\|_F$,
we next present a priori accuracy estimates for the Ritz approximations
computed by the SSLBD method.
To this end, we first establish the following two lemmas.

\begin{lemmamy}\label{lemma1}
Let $\mathscr{Q}_m$ be the orthogonal projector onto the subspace
$\UU_m\oplus\VV_m$. Then
	\begin{equation}\label{defThetajZj}
		(\Theta_j,Z_j)=\left(\begin{bmatrix}&\theta_j\\-\theta_j&\end{bmatrix},
		\begin{bmatrix}\tilde u_j& \tilde v_j\end{bmatrix}\right),\qquad
  j=1,2,\ldots,m
	\end{equation}
	are the block eigenpairs of the linear operator $\QQ_mA\QQ_m$
restricted to $\UU_m\oplus\VV_m$:
	\begin{equation}\label{blockeigen}
		\QQ_mA\QQ_mZ_j=Z_j\Theta_j.
	\end{equation}
\end{lemmamy}

\begin{proof}	
Since the columns of $[P_m,Q_m]$ form an orthonormal basis of $\UU_m\oplus\VV_m$,
the orthogonal projector onto $\UU_m\oplus\VV_m$ is $\QQ_m=[P_m,Q_m][P_m,Q_m]^T$.
By $\tilde u_j=P_mc_j$ and $\tilde v_j=Q_md_j$, we have
$$
Z_j=\begin{bmatrix}\tilde u_j& \tilde v_j\end{bmatrix}
   =\begin{bmatrix}P_m& Q_m\end{bmatrix}
    \begin{bmatrix}c_j&\\& d_j\end{bmatrix}.
$$
It is known from the proof of Theorem~\ref{thm3}
that $P_m^TAP_m=Q_m^TAQ_m=\bm{0}$. Making use of that and
$P_m^TAQ_m=B_m$, $Q_m^TAP_m=-B_m^T$ as well as \eqref{Ritz},
by $\QQ_mZ_j=Z_j$, we obtain
$$
\QQ_mA\QQ_mZ_j=\begin{bmatrix}P_m \!\!\!&\!\!\! Q_m\end{bmatrix}
\begin{bmatrix}& \!\!\!B_m\\-B^T_m\!\!\!&\end{bmatrix}
\begin{bmatrix}c_j\!\!\!&\\&\!\!\! d_j\end{bmatrix}
=\begin{bmatrix}P_m \!\!\!&\!\!\! Q_m\end{bmatrix}
\begin{bmatrix}c_j\!\!\!&\\&\!\!\! d_j\end{bmatrix}
\begin{bmatrix}&\!\!\! \theta_j\\-\theta_j\!\!\!&\end{bmatrix}
=Z_j\Theta_j.\qedhere
$$
\end{proof}

\begin{lemmamy}\label{lemma2}
With the notations of Lemma~\ref{lemma1}, for an arbitrary
$E\in\mathbb{R}^{2\times 2}$, it holds that
\begin{equation}\label{gapthetasigma}
	\qquad
	\|\Theta_jE-E\Lambda_i\|_F \geq |\theta_j-\sigma_i|\|E\|_F,
\end{equation}	
where $\Theta_j, j=1,\dots,m$ and $\Lambda_i,i=1,\dots,\ell$
are defined by \eqref{defThetajZj} and \eqref{defWL2}, respectively.
\end{lemmamy}

\begin{proof}
Since $\Theta_j=\theta_j\widehat I_2$ and
$\Lambda_i=\sigma_i\widehat I_2$
with $\widehat I_2=\bsmallmatrix{&1\\-1&}$,
by the fact that $\|\widehat I_2E\|_F=\|E\|_F$
and $\|E\widehat I_2\|_F=\|E\|_F$, we have
\begin{eqnarray}\label{gapthetasigma1}
\|\Theta_jE-E\Lambda_i\|_F&=&\sqrt{
\|\Theta_jE\|_F^2+\|E\Lambda_i\|_F^2-
2\trace(\Lambda_i^TE^T\Theta_jE)} \nonumber \\
&=&\sqrt{ \theta_j^2\|E\|_F^2+\sigma_i^2\|E\|_F^2-
2\theta_j\sigma_i\trace(\widehat I_2^TE^T\widehat I_2E)}.
\end{eqnarray}
Denote $E=\begin{bmatrix}e_{11}&e_{12}\\e_{21}&e_{22}\end{bmatrix}$.
Then the trace
\begin{equation*}
\trace(\widehat I_2^TE^T\widehat I_2E)=2e_{11}e_{22}-2e_{12}e_{21}\leq e_{11}^2+e_{22}^2+e_{12}^2+e_{21}^2 =\|E\|_F^2,
\end{equation*}
applying which to \eqref{gapthetasigma1} gives
\begin{equation}\label{gapthetasigma2}
\|\Theta_jE-E\Lambda_i\|_F\geq  \sqrt{\theta_j^2\|E\|_F^2+\sigma_i^2\|E\|_F^2-2\theta_j\sigma_i\|E\|_F^2}=
|\theta_j-\sigma_i|\|E\|_F. \qedhere
\end{equation}
\end{proof}

Making use of Lemmas~\ref{lemma1}--\ref{lemma2},
we are now in a position to establish a priori
accuracy estimates for the Ritz approximations.

\begin{theoremmy}\label{thm4plus}
For $(\sigma_j,u_j,v_j),j=1,\dots,\ell$, assume that
$(\theta_{j^{\prime}},\tilde u_{j^{\prime}},\tilde v_{j^{\prime}})$
is a Ritz approximation to the desired $(\sigma_j,u_j,v_j)$
where $\theta_{j^{\prime}}$ is the closest to $\sigma_j$
among the Ritz values.
Denote $\XX_j=\spans\{u_j,v_j\}$ and
$\ZZ_{j^{\prime}}=\spans\{\tilde u_{j^{\prime}},
\tilde v_{j^{\prime}}\}$. Then
	\begin{equation}\label{accuv}
	\|\sin\angle(\ZZ_{j^{\prime}},\XX_j)\|_F\leq
	\sqrt{1+\frac{\|\QQ_mA(I-\QQ_m)\|^2}
		{\min_{i\neq j^{\prime}}|\sigma_j-\theta_i|^2}}
	\|\sin\angle(\UU_m\oplus\VV_m,\mathcal{X}_j)\|_F,	
	\end{equation}
where $\QQ_m$ is the orthogonal projector onto the
subspace $\UU_m\oplus\VV_m$.
\end{theoremmy}

\begin{proof}
Notice that the two dimensional subspace
$\WW_j^{\prime}=\QQ_m\XX_j\subset\UU_m\oplus\VV_m$
is the orthogonal projection of $\XX_j$ onto $\UU_m\oplus\VV_m$
with respect to the $F$-norm. Therefore, by the definition of $\|\sin\angle(\UU_m\oplus\VV_m,\mathcal{X}_j)\|_F$, we have
\begin{equation}\label{tanWW1XXj}
\|\sin\angle(\WW_j^{\prime},\mathcal{X}_j)\|_F=
\|\sin\angle(\UU_m\oplus\VV_m,\mathcal{X}_j)\|_F. 	
\end{equation}

Let $W_1\in\mathbb{R}^{n\times 2}$ and
$W_2\in\mathbb{R}^{n\times(n-2m)}$
be the orthonormal basis matrices of $\WW_j^{\prime}$
and the orthogonal complement
of $\UU_m\oplus\VV_m$ with respect to $\mathbb{R}^n$,
respectively. The orthonormal basis matrix
$X_j=[u_j,v_j]$ of $\XX_j$ can be expressed as
\begin{equation}\label{decompXXj1}
		X_j=W_1H_1+W_2H_2,
	\end{equation}	
where $H_1\in\mathbb{R}^{2\times 2}$ and
$H_2\in\mathbb{R}^{(n-2m)\times2}$.
Since $AX_j=X_j\Lambda_j$, from the above relation we have
\begin{equation*}
	AW_1H_1-W_1H_1\Lambda_j=W_2H_2\Lambda_j-AW_2H_2.
\end{equation*}
Premultiplying this relation by $\QQ_m$ and taking the $F$-norms
in the two hand sides, from $\QQ_mW_1=W_1$,
$\QQ_mW_2=\bm{0}$ and $I-\QQ_m=W_2W_2^T$, we obtain
\begin{equation}\label{error2}
	\|\QQ_mAW_1H_1-W_1H_1\Lambda_j\|_F=
	\|\QQ_mAW_2H_2\|_F \leq
\|\QQ_mA(I-\QQ_m)\|\|H_2\|_F.	
\end{equation}

Recall the definition of $Z_{j^{\prime}}$ in \eqref{defThetajZj},
and denote $\widehat Z_{j^{\prime}}=[Z_1,\dots,
Z_{j^{\prime}-1},Z_{j^{\prime}+1},\dots,Z_m]$.
Note that the columns of $[Z_{j^{\prime}},\widehat Z_{j^{\prime}}]$
form an orthonormal basis of  $\UU_m\oplus\VV_m$.
Since $\WW_j^{\prime}\subset\UU_m\oplus\VV_m$, we can decompose
its orthonormal basis matrix $W_1$ into
\begin{equation}\label{decompWW1}
	W_1=Z_{j^{\prime}}F_1+\widehat Z_{j^{\prime}}F_2,
\end{equation}
where $F_1\in\mathbb{R}^{2\times 2}$ and
$F_2\in\mathbb{R}^{(2m-2)\times 2}$.
Lemma~\ref{lemma1} shows that
$\QQ_mAZ_{j^{\prime}}=Z_{j^{\prime}}\Theta_{j^{\prime}}$
and
$\QQ_mA\widehat Z_{j^{\prime}}=
\widehat Z_{j^{\prime}}\widehat\Theta_{j^{\prime}}$
with $\widehat\Theta_{j^{\prime}}=
\diag\{\Theta_1,\dots,\Theta_{j^{\prime}-1},
\Theta_{j^{\prime}+1},\dots,\Theta_m\}$.
Therefore, by~\eqref{decompWW1}, we obtain
\begin{eqnarray}
	\|\QQ_mAW_1H_1-W_1H_1\Lambda_j\|_F&=&
	\|\QQ_mAZ_{j^{\prime}}F_1H_1-Z_{j^{\prime}}F_1H_1\Lambda_j+
	\QQ_mA\widehat Z_{j^{\prime}}F_2H_1-
	\widehat Z_{j^{\prime}}F_2H_1\Lambda_j\|_F 	\nonumber \\
	&=&\|Z_{j^{\prime}}(\Theta_{j^{\prime}}F_1H_1-F_1H_1\Lambda_j)+
	\widehat Z_{j^{\prime}}(\widehat \Theta_{j^{\prime}}
F_2H_1-F_2H_1\Lambda_j)\|_F\nonumber \\
	&\geq&\|\widehat Z_{j^{\prime}}
	(\widehat \Theta_{j^{\prime}}F_2H_1-F_2H_1\Lambda_j\|_F\nonumber \\
	&=&\|\widehat \Theta_{j^{\prime}}F_2H_1-F_2H_1\Lambda_j\|_F,  \label{error3}
\end{eqnarray}
where the last two relations hold since
$[Z_{j^{\prime}},\widehat Z_{j^{\prime}}]$ is orthonormal.
Write
$$
E=F_2H_1=[E_1^T,\dots,E_{j^{\prime}-1}^T,E_{j^{\prime}+1}^T,E_{m}^T]^T
$$
with each $E_i\in\mathbb{R}^{2\times 2}$.
Then by Lemma~\ref{lemma2} we obtain
\begin{eqnarray}
	\|\widehat \Theta_{j^{\prime}}F_2H_1-F_2H_1\Lambda_j\|_F^2
	&=&\sum_{i=1,i\neq j^{\prime}}^{m}\|\Theta_{i}E_i-E_i\Lambda_j\|_F^2\geq	
	\sum_{i=1,i\neq j^{\prime}}^{m}|\theta_i-\sigma_j|^2\|E_i\|_F^2 \nonumber\\
	&\geq&
	\min_{i\neq j^{\prime}}|\theta_i-\sigma_j|^2\sum_{i=1,i\neq j^{\prime}}^{m}\|E_i\|_F^2
	=\min_{i\neq j^{\prime}}|\theta_i-\sigma_j|^2\|E\|_F^2
	\nonumber\\
	&=&\min_{i\neq j^{\prime}}|\theta_i-\sigma_j|^2\|F_2H_1\|_F^2.  \label{error4}
\end{eqnarray}
Therefore, combining \eqref{error2}, \eqref{error3} and \eqref{error4},
we obtain
\begin{equation}\label{error5}
	\|F_2H_1\|_F\leq\frac{\|\QQ_mA(I-\QQ_m)\|}
	{\min_{i\neq j^{\prime}}|\theta_i-\sigma_j|}\|H_2\|_F.
\end{equation}

Since both $X_j$ and $W_1$ are column orthonormal,
decomposition \eqref{decompXXj1} shows that
\begin{equation}\label{tanxjW1}
	\|\sin\angle(\WW_j^{\prime},\XX_j)\|_F=
	\|(I-W_1W_1^T)X_j\|_F=\|H_2\|_F.
\end{equation}
Substituting \eqref{decompWW1} into \eqref{decompXXj1} yields
\begin{equation}\label{decompXXj2}
	X_j=Z_{j^{\prime}}F_1H_1+\widehat Z_{j^{\prime}}F_2H_1+W_2H_2,	
\end{equation}
which, by the fact that the columns of
$[Z_{j^{\prime}},\widehat Z_{j^{\prime}},W_2]$
are orthonormal, means that
\begin{equation*}
\|\sin\angle(\ZZ_{j^{\prime}},\mathcal{X}_j)\|_F^2=
\|(I-Z_{j^{\prime}}Z_{j^{\prime}}^T)X_j\|_F^2=
\|\widehat Z_{j^{\prime}}F_2H_1+W_2H_2\|_F^2
=\|F_{2}H_1\|_F^2+\|H_2\|_F^2.	
\end{equation*}
Then \eqref{accuv} follows by in turn applying \eqref{error5},
\eqref{tanxjW1} and \eqref{tanWW1XXj} to the above relation.
\end{proof}

\begin{remark}
This theorem extends a classic result
(cf. \cite[pp.103, Theorem 4.6]{saad2011numerical}) on the a
priori error bound for the Ritz vector in the real symmetric
(complex Hermitian) case to the skew-symmetric case for the
Ritz block $\ZZ_{j^{\prime}}=\spans\{\tilde u_{j^{\prime}},
\tilde v_{j^{\prime}}\}$.
\end{remark}

\begin{theoremmy}\label{thm5}
With the notations of Theorem~\ref{thm4plus}, assume that the angles
\begin{equation}\label{assump}
\angle(\tilde{u}_{j^{\prime}},u_j)\leq \frac{\pi}{4}
\mbox{\ \ \ and\ \ \ }
\angle(\tilde{v}_{j^{\prime}},v_j)\leq \frac{\pi}{4}.
\end{equation}
Then
\begin{eqnarray}
|\theta_{j^{\prime}}-\sigma_j|&\leq& \sigma_j
\|\sin\angle(\ZZ_{j^{\prime}},\mathcal{X}_j)\|^2+
\frac{\sigma_1}{\sqrt{2}}\|\sin\angle(\ZZ_{j^{\prime}},\mathcal{X}_j)\|_F^2,
\label{ritzerror}\\
|\theta_{j^{\prime}}-\sigma_j|&\leq& (\sigma_1+\sigma_j)
\|\sin\angle(\ZZ_{j^{\prime}},\mathcal{X}_j)\|^2. \label{ritzerror2}
\end{eqnarray}
\end{theoremmy}

\begin{proof}
Decompose the orthonormal matrix $Z_{j^{\prime}}$ defined by
\eqref{defThetajZj} into the orthogonal direct sum
\begin{equation}\label{newexpressZj}
	Z_{j^{\prime}}=X_jH+\widehat X_j\widehat H,
\end{equation}
where $H\in\mathbb{R}^{2\times 2}$ and
$\widehat{H}\in\mathbb{R}^{(n-2)\times 2}$. Then
\begin{equation}\label{identity}
H^TH+\widehat H^T\widehat H=I_2.
\end{equation}

By \eqref{blockeigen}, \eqref{newexpressZj} and
$AX_j=X_j\Lambda_j$,
$A\widehat X_j=\widehat X_j\widehat \Lambda_j$, we obtain
\begin{equation}\label{diff}
\Theta_{j^{\prime}}-\Lambda_j=Z_{j^{\prime}}^TAZ_{j^{\prime}}-\Lambda_j
=
H^T\Lambda_{j}H
+\widehat{H}^T\widehat{\Lambda}_{j}\widehat{H}-\Lambda_j
=\sigma_j\cdot \left(H^T\widehat I_2 H-\widehat I_2\right)
+\widehat{H}^T\widehat{\Lambda}_{j}\widehat{H},	
\end{equation}
where the last equality holds since
$\Lambda_j=\sigma_j\widehat I_2$ with
$\widehat I_2 =\bsmallmatrix{&1\\-1}$.
Notice that $\Theta_{j^{\prime}}=\theta_{j^{\prime}}\widehat I_2$
and $\|\widehat I_2\|_F=\sqrt2$.
Taking the $F$-norms in the two hand sides of the above relation
and exploiting $\|\widehat \Lambda_{j}\|\leq\sigma_1$ and $\|\widehat H\|_F=\|\sin\angle(\ZZ_{j^{\prime}},\mathcal{X}_j)\|_F$ yield
\begin{eqnarray}
\sqrt{2}|\theta_{j^{\prime}}-\sigma_j|
&\leq&  \sigma_j\|H^T\widehat I_2H-\widehat I_2\|_F+
\|\widehat{\Lambda}_{j}\|\|\widehat{H}\|_F^2  \nonumber \\
&\leq& \sigma_j\|H^T\widehat I_2H-\widehat I_2\|_F+
\sigma_1 \|\sin\angle(\ZZ_{j^{\prime}},\mathcal{X}_j)\|_F^2. \label{interm1}
\end{eqnarray}

Write $H=[h_{ij}]$, and notice that
$h_{11}=\cos\angle(\tilde{u}_{j^{\prime}},u_j)$ and $h_{22}=\cos\angle(\tilde{v}_{j^{\prime}},v_j)$.
Then by \eqref{assump}
we have
\begin{equation}\label{hij}
	h_{11}\geq\frac{1}{\sqrt{2}}, \qquad
	h_{22}\geq\frac{1}{\sqrt{2}} \qquad\mbox{and}\qquad
	|h_{12}|\leq\frac{1}{\sqrt{2}}, \qquad
	|h_{21}|\leq\frac{1}{\sqrt{2}}
\end{equation}
since, from \eqref{identity}, the 2-norms of columns
$[h_{11},h_{21}]^T$ and $[h_{12},h_{22}]^T$ of $H$
are no more than one. As a consequence, we obtain
\begin{eqnarray*}
{\rm det}(H)&=&h_{11}h_{22}-h_{12}h_{21}\geq
h_{11}h_{22}-|h_{12}h_{21}|\\
&\geq &\frac{1}{\sqrt 2}\cdot\frac{1}{\sqrt 2}-
 \frac{1}{\sqrt 2}\cdot\frac{1}{\sqrt 2}
= 0.
\end{eqnarray*}
Let $\sigma_1(H)\geq \sigma_2(H)$ be the singular
values of $H$. Therefore, by definition we have
$
\sigma_2(H)=\|\cos\angle(\ZZ_{j^{\prime}},\mathcal{X}_j)\|
$
and
\begin{equation*}
	{\rm det}(H)=\sigma_1(H)\sigma_2(H)\geq \sigma_2^2(H)
	=\|\cos\angle(\ZZ_{j^{\prime}},\mathcal{X}_j)\|^2.
\end{equation*}
Then
it is straightforward to verify that
\begin{eqnarray}
	\|H^T\widehat I_2H-\widehat I_2\|_F&=&\left\|
	\begin{bmatrix}
		&h_{11}h_{22}-h_{12}h_{21}-1\\ h_{12}h_{21}-h_{11}h_{22}+1&
	\end{bmatrix}\right\|_F    \nonumber
	\\&=&\sqrt2 (1-{\rm det}(H))\leq\sqrt{2}
	(1-\|\cos\angle(\ZZ_{j^{\prime}},\mathcal{X}_j)\|^2)
	\nonumber \\&\leq&
	\sqrt2\|\sin\angle(\ZZ_{j^{\prime}},\mathcal{X}_j)\|^2. \nonumber
\end{eqnarray}
Applying the above inequality to the right hand side
of \eqref{interm1}, we obtain
$$
|\theta_{j^{\prime}}-\sigma_j|\leq \sigma_j
\|\sin\angle(\ZZ_{j^{\prime}},\mathcal{X}_j)\|^2+
\frac{\sigma_1}{\sqrt{2}}\|\sin\angle(\ZZ_{j^{\prime}},
\mathcal{X}_j)\|_F^2.
$$
Similarly, we get bound \eqref{ritzerror2} from \eqref{diff}
by using $\|\widehat{H}\|=\|\sin\angle(\ZZ_{j^{\prime}},
\mathcal{X}_j)\|$ and
$$
\|H^T\widehat I_2H-\widehat I_2\|
=1-{\det}(H)\leq\|\sin\angle(\ZZ_{j^{\prime}},\mathcal{X}_j)\|^2. \qedhere
$$
\end{proof}

We can amplify bound \eqref{ritzerror} as
$$
|\theta_{j^{\prime}}-\sigma_j|\leq \left(1+
\frac{1}{\sqrt{2}}\right)\sigma_1
\|\sin\angle(\ZZ_{j^{\prime}},\mathcal{X}_j)\|_F^2.
$$
Applying \eqref{accuv} to $\|\sin\angle(\ZZ_{j^{\prime}},
\mathcal{X}_j)\|_F$,
we are able to estimate how fast
$|\theta_{j^{\prime}}-\sigma_j|$ tends to zero as $m$
increases. Since $\XX_j$ approaches
$\UU_m\oplus\VV_m$ for a fixed small $j\ll m$ as
$m$ increases, if $\sigma_j$ is well
separated from the Ritz values of $A$ other than the approximate
singular value $\theta_{j^{\prime}}$,
that is, $\min_{i\neq j^{\prime}}|\sigma_j-\theta_i|$ is not small, then
$\ZZ_{j^{\prime}}$ converges to $\XX_j$ as fast as $\XX_j$ tends to
$\UU_m\oplus\VV_m$. In this case,
Theorem~\ref{thm5} shows that the convergence of
$\theta_{j^{\prime}}$ to $\sigma_j$ is quadratic relative to
that of $\ZZ_{j^{\prime}}$. We should point out that
assumption \eqref{assump} is weak and is met very soon as $m$ increases.

For brevity, we now omit the subscript $j$ and denote by
$\left(\theta,\tilde u,\tilde v\right)$
a Ritz approximation at the $m$th step.
We can then recover
two conjugate approximate eigenpairs from $(\theta,\tilde u,\tilde v)$
in the form of $(\tilde\lambda_{\pm},\tilde x_{\pm})
=(\pm\mi\theta,\frac{1}{\sqrt2}(\tilde u\pm\mi\tilde v))$.
Note that their residuals are
\begin{equation}\label{defres}
	r_{\pm} = A\tilde x_{\pm}-\tilde\lambda_{\pm} \tilde x_{\pm}
	= r_{\mathrm{R}} \pm \mi r_{\mathrm{I}}
	\qquad\mbox{with}\qquad
	\left\{\begin{aligned}
		r_{\mathrm{R}}&=\tfrac{1}{\sqrt2}(A\tilde u+\theta\tilde v),\\
		r_{\mathrm{I}}&=\tfrac{1}{\sqrt2}(A\tilde v-\theta\tilde u).
	\end{aligned}
	\right.
\end{equation}
Obviously, $r_{\pm}=\bm{0}$ if and only if $(\theta,\tilde u,\tilde v)$
is an exact singular triplet of $A$.
We claim that $(\tilde\lambda_{\pm},\tilde x_{\pm})$ has
converged if its residual norm satisfies
\begin{equation}\label{res}
  \|r_{\pm}\|=\sqrt{\|r_{\mathrm{R}}\|^2+
  	\|r_{\mathrm{I}}\|^2}\leq \|A\|\cdot tol,
\end{equation}
where $tol>0$ is a prescribed tolerance.
In practical computations,
one can replace $\|A\|$ by the largest Ritz value $\theta_1$.
Notice that
$$
\sqrt{2} \|r_{\pm}\|=\sqrt{\|A\tilde u+\theta\tilde v\|^2
	+\|A\tilde v-\theta\tilde u\|^2}
$$
is nothing but the residual norm of
the Ritz approximation $(\theta,\tilde u,\tilde v)$ of $A$.
Therefore, the eigenvalue problem and SVD problem of $A$ essentially
share the same general-purpose stopping criterion.

By inserting $\tilde u=P_mc$ and $\tilde v=Q_md$ into \eqref{defres} and
making use of \eqref{LBDmat} and \eqref{Ritz}, it is easily justified that
\begin{equation}\label{resieasy}
	\|r_{\pm}\|=\tfrac{1}{\sqrt2}\sqrt{\|B_m^Tc-\theta d\|^2+
\gamma_m^2|e_{m,m}^Tc|^2+\|B_md-\theta c\|^2}
=\tfrac{1}{\sqrt2}\gamma_m|e_{m,m}^Tc|.
\end{equation}
This indicates that we can calculate the residual norms of
the Ritz approximations cheaply without
explicitly forming the approximate singular vectors. We only
compute the approximate singular
vectors of $A$ until the corresponding residual norms defined
by \eqref{resieasy} drop below a prescribed tolerance.
Moreover, \eqref{resieasy} shows that once the $m$-step
SSLBD process breaks down, i.e., $\gamma_m=0$, all the
computed  $m$ approximate singular triplets
$(\theta,\tilde u,\tilde v)$ are the exact ones of $A$,
as we have mentioned at the end of Section~\ref{subsec:1}.

\section{An implicitly restarted SSLBD algorithm
	with partial reorthogonalization}\label{sec:4}
This section is devoted to efficient partial reorthogonalization
and the development of an implicitly restarted SSLBD algorithm,
which are crucial for a practical purpose
and the numerical reliability so as to
avoid ghost phenomena of computed Ritz approximations
and make the algorithm behave as if it is in exact arithmetic.

\subsection{Partial reorthogonalization}\label{subsec:5}

We adopt a commonly used two-sided partial reorthogonalization
strategy, as done in \cite{jia2003implicitly,jia2010refined,
	larsen2001combining}, to make
the columns of $P_m$ and $Q_{m+1}$ numerically
orthogonal to the level $\mathcal{O}(\sqrt{\varepsilon})$ with $\varepsilon$
the machine precision. In finite precision arithmetic, however,
such semi-orthogonality cannot automatically guarantee the
numerical semi-biorthogonality of $P_m$ and $Q_m$,
causing that computed Ritz values have duplicates or ghosts
and the convergence is delayed.
To avoid this severe deficiency and make the method as if it does in
exact arithmetic, one of our particular concerns
is to design an effective and efficient partial reorthogonalization
strategy, so that the entries of $Q_m^TP_m$ are
$\mathcal{O}(\sqrt{\varepsilon})$ in size; that is, the columns
of $P_m$ and $Q_m$ are numerically semi-biorthogonal.
As it will turn out, how to achieve this goal efficiently is involved.
Why the semi-orthogonality and semi-biorthogonality suffice
is based on the following fundamental result.

\begin{theoremmy}\label{thm7}
Let the bidiagonal $B_m$ and
$P_m=[p_1,\ldots,p_m]$ and $Q_{m}=[q_1,\ldots,q_{m}]$ be produced
by the $m$-step SSLBD process.
Assume that $\beta_i,\ i=1,2,\ldots,m$ and $\gamma_i,\ i=1,2,\ldots,m-1$
defined in \eqref{LBDmat} are not small, and
	\begin{equation}\label{semiorth}
		\max\left\{\max_{1\leq i<j\leq m} \left|p_i^Tp_j\right|,
		\max_{1\leq i<j\leq m} \left|q_i^Tq_j\right|,
		\max_{1\leq i , j\leq m} \left|p_i^Tq_j\right|\right\}
		\leq\sqrt{\frac{\epsilon}{m}}.
	\end{equation}
Let $[\check q_1,\check p_1,\dots,\check q_m,\check p_m]$
be the Q-factor in the QR decomposition of $M_m=[q_1,p_1,\dots,q_m,p_m]$
with the diagonals of the upper triangular matrix being positive, and define
the orthonormal matrices $\check P_m=[\check p_1,\dots,\check p_m]$ and
$\check Q_{m}=[\check q_1,\dots,\check q_{m}]$.
Then
\begin{equation}\label{pqm}
	\check P_m^TA\check Q_{m}=  B_m+\Delta_m,
\end{equation}
where the entries of $\Delta_m$ are $\mathcal{O}(\|A\|\epsilon)$ in size.
\end{theoremmy}
\begin{proof}
Recall the equivalence of decompositions \eqref{bidiagA} and
\eqref{defH}, and notice that decomposition \eqref{defH} can
be realized by the skew-symmetric Lanczos process, and
\eqref{bidiagA} can be computed by our SSLBD process.
The result then follows from a proof analogous to that of Theorem 4
in \cite{simon1984analysis}, and we thus omit details.
\end{proof}

Theorem~\ref{thm7} shows that if the computed Lanczos vectors
are semi-orthogonal and semi-biorthogonal
then the upper bidiagonal matrix $B_m$ is, up to roundoff error,
equal to the projection matrix $\check P_m^TA\check Q_{m}$
of $A$ with respect
to the left and right searching subspaces $\UU^{\prime}_m
=\mathcal{R}(\check{P}_m)$ and $\VV^{\prime}_m=\mathcal{R}(\check{Q}_m)$.
This means that the singular values of $B_m$,
the computed Ritz values, are as accurate as those of
$\check P_m^TA\check Q_{m}$, i.e., the exact Ritz values,
within the level $\mathcal{O}(\|A\|\epsilon)$.
In other words, the semi-orthogonality and
semi-biorthogonality suffice to guarantee that the computed and true
Ritz values are equally accurate in finite precision arithmetic.
Therefore, the full numerical orthogonality
and biorthogonality at the level $\mathcal{O}(\varepsilon)$ do not
help and instead cause unnecessary waste
for the accurate computation of singular values.

To achieve the desired numerical semi-orthogonality and
semi-biorthogonality, we first need to efficiently estimate
the levels of orthogonality and biorthogonality among all
the Lanczos vectors to monitor if \eqref{semiorth} is guaranteed.
Whenever \eqref{semiorth} is violated,
we will use an efficient partial reorthogonalization
to restore \eqref{semiorth}. To this end,
we adapt the recurrences in Larsen's PhD thesis \cite{larsen1998lanczos}
to the effective estimations of the desired semi-orthogonality
and semi-biorthogonality, as is shown below.

Denote by $\Phi\in\mathbb{R}^{m\times m}$ and
$\Psi\in\mathbb{R}^{(m+1)\times(m+1)}$ the estimate
matrices of orthogonality
between the columns of $P_m$ and $Q_{m+1}$,
respectively, and by $\Omega\in\mathbb{R}^{m\times (m+1)}$
the estimate matrix of biorthogonality of
$P_m$ and $Q_{m+1}$.
That is, their $(i,j)$-elements $\varphi_{ij}\approx p_i^Tp_j$,
$\psi_{ij}\approx q_i^Tq_j$ and $\omega_{ij}\approx p_i^Tq_j$, respectively.
Adopting the recurrences in \cite{larsen1998lanczos},
we set $\varphi_{ii}=\psi_{ii}=1$ for $i=1,\dots,m$,
and at the $j$th step use the following recurrences to compute
\begin{eqnarray}
\qquad\qquad	&&\left\{\begin{aligned}
		&\varphi_{ij}^{\prime}=\beta_{i}\psi_{ij}+\gamma_i
		\psi_{i+1,j}-\gamma_{j-1}\varphi_{i,j-1},\\
		&\varphi_{ij}=(\varphi_{ij}^{\prime}+
		sign(\varphi_{ij}^{\prime})\epsilon_1)/\beta_{j},
	\end{aligned}
	\right.\quad\qquad{\ \!}\mbox{for}\quad i=1,\dots,j-1, \label{orthPP}
	\\[0.3em]
\qquad\qquad	&&\left\{\begin{aligned}
		&\psi_{i,j+1}^{\prime}=\gamma_{i-1}\varphi_{i-1,j}
		+\beta_i\varphi_{ij}-\beta_j\psi_{ij},\\
		&\psi_{i,j+1}=(\psi_{i,j+1}^{\prime}
		+sign(\psi_{i,j+1}^{\prime})\epsilon_1)/\gamma_j,
	\end{aligned}
	\right.\qquad\mbox{for}\quad i=1,\dots,j, \label{orthQQ}
\end{eqnarray}
where $sign(\cdot)$ is the sign function,
$\epsilon_1=\frac{\epsilon\sqrt{n}\|A\|_e}{2}$
and $\|A\|_e$ is an estimate for $\|A\|$.
The terms $sign(\cdot)\epsilon_{1}$ in \eqref{orthPP} and \eqref{orthQQ}
take into account the rounding errors in LBD,
which makes the calculated items as large as possible in magnitude
so that it is safe to take them as estimates for the levels of orthogonality
among the left and right Lanczos vectors.
Based on the sizes of
$\varphi_{ij}$ and $\psi_{i,j+1}$, one can decide to which previous
Lanczos vectors the newly computed one should be reorthogonalized.

Making use of the skew-symmetry of $A$,
we can derive analogous recurrences to those in \cite{larsen1998lanczos}
to compute the elements $\omega_{ij}$ of $\Omega$.
The derivations are rigorous but tedious, so we omit details here, and present
the recurrences directly.
Initially, set $\omega_{i,0}=0$ and $\omega_{0,j}=0$ for $i=1,\dots,m$ and
$j=1,\dots,m+1$. At the $j$th step, we compute the elements of $\Omega$ by
\begin{eqnarray}
\hspace{3.5em}	&&\left\{\begin{aligned}
		&\omega_{ji}^{\prime}=\beta_{i}\omega_{ij}
		+\gamma_{i-1}\omega_{i-1,j}
		+\gamma_{j-1}\omega_{j-1,i},\\
		&\omega_{ji}=-(\omega_{ji}^{\prime}
		+sign(\omega_{ji}^{\prime})\epsilon_1)/\beta_{j},
	\end{aligned}
	\right.\ \ \ {\ \!}\qquad\mbox{for}\quad i=1,\dots,j, \label{orthPQ}
	\\[0.3em]
	&&\left\{\begin{aligned}
		&\omega_{i,j+1}^{\prime}=\gamma_{i}\omega_{j,i+1}
		+\beta_i\omega_{ji}+\beta_{j}\omega_{ij},\\
		&\omega_{i,j+1}=-(\omega_{i,j+1}^{\prime}
		+sign(\omega_{i,j+1}^{\prime})\epsilon_1)/\gamma_j,
	\end{aligned}
	\right.\qquad\mbox{for}\quad i=1,\dots,j. \label{orthQP}
\end{eqnarray}

Having computed $\Phi$, $\Psi$ and $\Omega$,
at the $j$th step of the SSLBD process,
we determine the index sets:
\begin{eqnarray}
\quad &&\mathcal{I}_{P}^{(j)}=\left\{i|1\leq i< j,
	\ |\varphi_{ij}|\geq\sqrt{\tfrac{\epsilon}{m}}\right\}, \hspace{2.6em}
	\mathcal{I}_{P,Q}^{(j)}=\left\{i|1\leq i\leq j,
	\ |\omega_{ji}|\geq\sqrt{\tfrac{\epsilon}{m}}\right\},  \label{orthPPQ}\\
\quad &&\mathcal{I}_{Q}^{(j)}=\left\{i|1\leq i\leq j,
	\ |\psi_{i,j+1}|\geq\sqrt{\tfrac{\epsilon}{m}}\right\},\quad\
	\mathcal{I}_{Q,P}^{(j)}=\left\{i|1\leq i\leq j,
	\  |\omega_{i,j+1}|\geq\sqrt{\tfrac{\epsilon}{m}}\right\}.  \label{orthQQP}
\end{eqnarray}
These four sets consist of the indices that correspond to the left and right
Lanczos vectors $p_i$ and $q_i$ that have lost the semi-orthogonality
and semi-biorthogonality with $p_j$ and $q_{j+1}$, respectively.
We then make use of the modified Gram--Schmidt (MGS) orthogonalization
procedure \cite{golub2012matrix,saad2003iterative,stewart1998matrix} to
reorthogonalize the newly computed
left Lanczos vector $p_j$ first against the left Lanczos vector(s)
$p_i$ with $i\in\mathcal{I}_{P}^{(j)}$ and then the right one(s)
$q_i$ with $i\in\mathcal{I}_{P,Q}^{(j)}$.
Having done these, we reorthogonalize the right Lanczos vector
$q_{j+1}$, which has newly been computed using the updated $p_j$,
first against the right Lanczos vector(s) $q_{i}$ for $i\in\mathcal{I}_{Q}^{(j)}$
and then the left one(s) $p_i$ for $i\in\mathcal{I}_{Q,P}^{(j)}$.

\begin{algorithm}[tbph]
	\caption{Partial reorthogonalization in the SSLBD process.}\label{alg3}
	\begin{itemize}[leftmargin = 2em]
		\item[\footnotesize (\ref{skwbld1}.1)\label{reorth31}]
Calculate $\varphi_{1j},\varphi_{2j},\dots,\varphi_{j-1,j}$
and $\omega_{j1},\omega_{j2},\dots,\omega_{jj}$ by
\eqref{orthPP} and \eqref{orthPQ}, respectively, and determine
the index sets $\mathcal{I}_{P}^{(j)}$ and $\mathcal{I}_{P,Q}^{(j)}$
according to  \eqref{orthPPQ}.
		
\item[\footnotesize (\ref{skwbld1}.2)] \textbf{for}
$i\in\mathcal{I}_{P}^{(j)}$ \textbf{do}
\hspace{13em}  \textcolor{gray}{$\%$ partially
	reorthogonalize $p_j$ to $P_{j-1}$}
\begin{itemize}[leftmargin=1.5em]
\item[] Calculate $\tau=p_i^Ts_j$, and update $s_j=s_j-\tau p_i$.
\item[] Overwrite $\varphi_{lj}^{\prime} =
\varphi_{lj}^{\prime}-\tau\varphi_{li}$ for $l=1,\dots,j-1$
and $\omega_{jl}^{\prime}= \omega_{jl}^{\prime}
-\tau\omega_{il}$ for $l=1,\dots,j$.
\end{itemize}
\item[] \textbf{end for}
			
\item[\footnotesize (\ref{skwbld1}.3)] \textbf{for}
$i\in\mathcal{I}_{P,Q}^{(j)}$ \textbf{do}  \hspace{13em}
 \textcolor{gray}{$\%$ partially reorthogonalize $p_j$ to $Q_{j}$}
\begin{itemize}[leftmargin=1.5em]
\item[] Compute $\tau =q_i^Ts_j$, and update $s_j=s_j-\tau q_i$.
\item[] Overwrite
$\omega_{jl}^{\prime} =\omega_{jl}^{\prime} -\tau\psi_{il}$ for $l=1,\dots,j$ and
$\varphi_{lj}^{\prime} =\varphi_{lj}^{\prime} -\tau\omega_{li}$ for $l=1,\dots,j-1$.
\end{itemize}
\item[] \textbf{end for}
\item[\footnotesize (\ref{skwbld1}.4)] Compute $\beta_j=\|s_j\|$,
and update $\varphi_{1j},\varphi_{2j},\dots,\varphi_{j-1,j}$ and
$\omega_{j1},\omega_{j2},\dots,\omega_{jj}$ by the second
relations in \eqref{orthPP} and \eqref{orthPQ}, respectively.
		
\vspace{0.5em}
		
\item[\footnotesize (\ref{skwbld3}.1)] Compute
$\psi_{1j},\psi_{2j},\dots,\psi_{jj}$ and
$\omega_{1j},\omega_{2j},\dots,\omega_{jj}$ using
\eqref{orthQQ} and \eqref{orthQP}, respectively,
and determine the index sets $\mathcal{I}_{Q}^{(j)}$
and $\mathcal{I}_{Q,P}^{(j)}$ according to \eqref{orthQQP}.
		
\item[\footnotesize (\ref{skwbld3}.2)] \textbf{for}
$i\in\mathcal{I}_{Q}^{(j)}$ \textbf{do} \hspace{13em}
\textcolor{gray}{$\%$ partially reorthogonalize
	$q_{j+1}$ to $Q_{j}$}
\begin{itemize}[leftmargin=1.5em]
\item[] Calculate $\tau=q_i^Tt_j$, and update $t_j=t_j-\tau q_i$.
\item[] Overwrite
$\psi_{l,j+1}^{\prime}=\psi_{l,j+1}^{\prime}-\tau \psi_{li}$ and $\omega_{l,j+1}^{\prime}=\omega_{l,j+1}^{\prime}
-\tau\omega_{li}$ for $l=1,\dots,j$.
\end{itemize}
\item[] \textbf{end for}
 		
\item[\footnotesize (\ref{skwbld3}.3)] \textbf{for}
$i\in\mathcal{I}_{Q,P}^{(j)}$ \textbf{do}
\hspace{13em} \textcolor{gray}{$\%$ partially
	reorthogonalize $q_{j+1}$ to $P_{j}$ }
\begin{itemize}[leftmargin=1.5em]
\item[] Compute $\tau=p_i^Tt_j$, and update $t_j=t_j-\tau p_i$.
\item[] Overwrite
$\omega_{l,j+1}^{\prime}=\omega_{l,j+1}^{\prime}-\tau \varphi_{li}$ and $\psi_{l,j+1}^{\prime}=\psi_{l,j+1}^{\prime}-\tau \omega_{il}$ for $l=1,\dots,j$.
\end{itemize}
\item[] \textbf{end for}
\item[\footnotesize (\ref{skwbld3}.4)] Calculate $\gamma_j=\|t_j\|$, and update $\omega_{1,j+1},\omega_{2,j+1},\dots,\omega_{j,j+1}$ and
$\psi_{1,j+1},\psi_{2,j+1},\dots,\psi_{j,j+1}$ by the second relations
in \eqref{orthQQ} and \eqref{orthQP}, respectively.
		
\end{itemize}
\end{algorithm}

Remarkably, once a Lanczos vector is reorthogonalized to a previous one,
the relevant elements in $\Phi$, $\Psi$ and $\Omega$ change
correspondingly, which may no longer be reliable estimates for
the levels of orthogonality and biorthogonality among the updated Lanczos
vector and all the previous ones.
To this end, we need to update those changed values of $\Phi$, $\Psi$
and $\Omega$ by the MGS procedure\footnote{In his PhD
thesis \cite{larsen1998lanczos},
Larsen resets those quantities as $\mathcal{O}(\epsilon)$
and includes the $\mathcal{O}(\epsilon)$ parts in $\epsilon_1$ of
\eqref{orthPP}--\eqref{orthQP}.  But he does not
explain why those quantities can automatically remain
$\mathcal{O}(\epsilon)$ as the procedure is going on.
In our procedure, we guarantee those vectors to be
numerically orthogonal and biorthogonal by
explicitly performing the MGS procedure.}.
Since $\Phi$ and $\Psi$ are symmetric matrices with diagonals being ones,
it suffices to compute their strictly upper triangular parts.
Algorithm~\ref{alg3} describes the whole partial reorthogonalization process
in SSLBD, where we insert steps~(3.1)--(3.4)
and (5.1)--(5.4) between steps~\ref{skwbld1}--\ref{skwbld2}
and steps~\ref{skwbld3}--\ref{skwbld4} of Algorithm~\ref{alg1}, respectively.
In such a way, we have developed an $m$-step
SSLBD process with partial reorthogonalization.

We see from Algorithm~\ref{alg3} and relations
\eqref{orthPP}--\eqref{orthQP} that
the $j$th step takes $\mathcal{O}(j)$ flops to compute the
new elements of $\Phi$, $\Psi$ and $\Omega$ for the first time,
which is negligible relative to the cost of performing $2$ matrix-vector
multiplications with $A$ at that step.
Also, it takes $\mathcal{O}(j)$ extra flops to update the elements of
$\Phi$, $\Psi$ and $\Omega$ after one step of reorthogonalization.
Therefore, we can estimate the levels of orthogonality
and biorthogonality very efficiently.

\subsection{An implicitly restarted algorithm}\label{subsec:6}

When the dimension of the searching subspaces reaches the maximum number $m$
allowed, implicit restart needs to select certain $m-k$ shifts.
The shifts can particularly be taken as the unwanted Ritz values
$\mu_1=\theta_{k+1},\dots, \mu_{m-k}=\theta_{m}$, called the exact shifts
\cite{larsen2001combining,sorensen1992implicit}.
For those so-called bad shifts, i.e., those close to some of the desired
singular values, we replace them with certain more appropriate values.
As suggested in \cite{larsen1998lanczos} and adopted
in \cite{jia2003implicitly,jia2010refined} with
necessary modifications when computing the smallest singular triplets,
we regard a shift $\mu$ as a bad one and reset it to be zero if
\begin{equation}\label{badshift}
	\left|(\theta_k-\|r_{\pm k}\|)-\mu\right|\leq{\theta_k} \cdot 10^{-3},
\end{equation}
where $\|r_{\pm k}\|=\sqrt{\|r_{\mathrm{R},k}\|^2+\|r_{\mathrm{I},k}\|^2}$
is the residual norm of the approximate singular triplet
$(\theta_k,\tilde u_k,\tilde v_k)$ with $r_{\mathrm{R},k}$
and $r_{\mathrm{I},k}$ defined by \eqref{defres}.

Implicit restart formally performs $m-k$ implicit QR iterations \cite{saad2011numerical,stewart2001matrix}
on $B_m^TB_m$ with the shifts $\mu_1^2,\dots,\mu_{m-k}^2$
by equivalently acting Givens rotations directly on $B_m$
and accumulating the orthogonal matrices
$\widetilde C_m$ and $\widetilde D_m$ with order $m$ such that
\begin{equation*}
	\left\{\begin{aligned}
		\ \ &\widetilde C_m^T(B_m^TB_m-\mu_{m-k}^2 I)\cdots(B_m^TB_m-\mu_1^2 I) \qquad\mbox{is}\qquad \mbox{upper triangular},\\[0.5em]
		&\widetilde B_m\ =\ \widetilde C_m^TB_m
		\widetilde D_m \hspace{10.3em} \mbox{is}\qquad \mbox{upper bidiagonal}.
	\end{aligned}
	\right.
	\end{equation*}
Using these matrices, we first compute
\begin{equation}\label{restart1}
B_{k,\new}=\widetilde B_{k},\qquad
P_{k,\new}=P_m\widetilde C_{k},\qquad
Q_{k,\new}=Q_m\widetilde D_{k},	
\end{equation}
where $\widetilde B_{k}$ is the $k$th leading principal submatrix of
$\widetilde B_m$, and $\widetilde C_{k}$ and $\widetilde D_{k}$ consist of the
first $k$
columns of $\widetilde C_m$ and $\widetilde D_m$, respectively. We then compute
\begin{equation}\label{restart3}
	q_{k+1}^{\prime}=\beta_{m}\tilde c_{mk} \cdot q_{m+1}
	+\tilde \gamma_k Q_m\tilde d_{k+1},\qquad
	\beta_{k,\new}=\|q^{\prime}_{k+1}\|,\qquad
	q_{k+1,\new}=\frac{q_{k+1}^{\prime}}{\beta_{k,\mathrm{new}}}
\end{equation}
with $\tilde d_{k+1}$ the $(k+1)$st column of $\widetilde D_m$
and $\tilde c_{mk}$, $\tilde \gamma_k$ the $(m,k)$- and
$(k,k+1)$-elements of $\widetilde C_m$ and $\widetilde B_m$, respectively.
As a result, when the above implicit restart is run,
we have already obtained a $k$-step SSLBD process rather than over scratch.
We then perform $(m-k)$-step SSLBD
process from step $k+1$ onwards to obtain a new $m$-step SSLBD process.

Once the SSLBD process is restarted, the matrices
$\Phi$, $\Psi$ and $\Omega$ must be updated.
This can be done efficiently by making use of
\eqref{restart1} and \eqref{restart3}.
Note that $\Phi$ and $\Psi$ are symmetric matrices with orders
$m$ and $m+1$, respectively, and
$\Omega$ is an $m\times (m+1)$ flat matrix.
The corresponding updated matrices are with orders $k$, $k+1$ and
size $k\times(k+1)$, respectively, which are updated by
$$
\Phi:=\widetilde C_k^T\Phi  \widetilde C_k,\qquad
\Psi_k:=\widetilde D_k^T\Psi_m \widetilde D_k,\qquad
\Omega_k:=\widetilde C_k^T\Omega_m\widetilde D_k,
$$
and
\begin{eqnarray*}	
	\Psi_{1:k,k+1}&:=&\widetilde D_k^T(\beta_m\tilde c_{mk}\Psi_{1:m,m+1}
	+ \tilde \gamma_{k}\Psi_{m}\tilde d_{k+1})/\beta_{k,\new}, \\
	\Omega_{1:k,k+1}&:=&\widetilde C_k^T(\beta_m\tilde c_{mk}\Omega_{1:m,m+1}
	+ \tilde \gamma_{k}\Omega_{m}\tilde d_{k+1})/\beta_{k,\new}.
\end{eqnarray*}
Here we have denoted by $(\cdot)_{j}$ and $(\cdot)_{1:j,j+1}$
the $j\times j$ leading principal submatrix and the first $j$ entries
in the $(j+1)$th
column of the estimate matrices of orthogonality and biorthogonality,
respectively. The updating process needs $\mathcal{O}(km^2)$ extra flops, which
is negligible relative to the $m$-step implicit restart for $m\ll \ell=n/2$.

As we have seen previously, the convergence criterion \eqref{res}
and the efficient partial reorthogonalization scheme
(cf.~\eqref{orthPP}--\eqref{orthQP})
depend on a rough estimate for $\|A\|$, which itself is approximated by the
largest approximate singular value $\theta_1$ in our algorithm.
Therefore, we simply set $\|A\|_e=\theta_1$ to replace $\|A\|$ in \eqref{res}.
Notice that, before the end of the first cycle of the
implicitly restarted algorithm, there is no approximate
singular value available.
We adopt a strategy analogous to that in \cite{larsen1998lanczos}
to estimate $\|A\|$: Set $\|A\|_e=\sqrt{\beta_1^2+\gamma_1^2}$ for $j=1$,
and update it for $j\geq 2$ by
\begin{equation*}
	\|A\|_e=\max\left\{\|A\|_e,
	\sqrt{\beta_{j-1}^2+\gamma_{j-1}^2
		+\gamma_{j-1}\beta_j+\gamma_{j-2}\beta_{j-1}},
	\sqrt{\beta_j^2+\gamma_{j}^2+\gamma_{j-1}\beta_j} \right\},
\end{equation*}
where $\gamma_{0}=0$.

\begin{algorithm}[tbph]
	\caption{Implicitly restarted SSLBD algorithm with partial reorthogonalization}
	\renewcommand{\algorithmicrequire}{\textbf{Input: }}
	\renewcommand{\algorithmicensure} {\textbf{Output: }}	
	\begin{algorithmic}[1]\label{alg2}
		\STATE{Initialization: Set $i=0$, $P_0=[\ \ ]$, $Q_1=[q_1]$ and $B_0=[\ \ ]$.}
		\WHILE{not converged or $i<i_{max}$}
		
		\STATE {For $i=0$ and $i>0$, perform the $m$- or $(m-k)$-step SSLBD
			process on $A$ with partial reorthogonalization to obtain $P_m$, $Q_{m+1}$ and $\widehat B_m=[B_m,\gamma_me_m]$.}
		
		\STATE{Compute the SVD \eqref{Ritz} of $B_m$ to obtain its singular
triplets $\left(\theta_j,c_j,d_j\right)$, and calculate the corresponding residual
			norms $\|r_{\pm j}\|$ by \eqref{resieasy}, $j=1,\dots,m$.}
			
		\STATE{\textbf{if} $\|r_{\pm j}\|\leq \|A\|_e \cdot tol$ for all $j=1,\dots,k$, \textbf{then}
compute $\tilde u_j=P_mc_j$ and $\tilde v_j=Q_md_j$ for $j=1,\dots,k$, and break the loop.}
		
		\STATE{Set $\mu_j=\theta_{j+k},j=1,\dots,m-k$,
			and replace those satisfying \eqref{badshift} with zeros.
Perform the implicit restart scheme with the shifts  $\mu_{1}^2,\dots,\mu_{m-k}^2$
			to obtain the updated $P_k$, $Q_{k+1}$ and $\widehat B_k$. Set $i=i+1$,
and goto step 3.}

		\ENDWHILE
	\end{algorithmic}
\end{algorithm}

Algorithm~\ref{alg2} sketches the implicitly restarted
SSLBD algorithm with partial reorthogonalization for computing the
$k$ largest singular triplets of $A$.
It requires the device to compute $A\underline{x}$
for an arbitrary $n$-dimensional vector $\underline{x}$
and the number $k$ of the desired conjugate eigenpairs of $A$.
The number $m$ is the maximum subspace dimension, $i_{\max}$
is the maximum restarts allowed, $tol$ is the stopping
tolerance in \eqref{res}, and the unit-length vector $q_1$
is an initial right Lanczos vector.
The defaults of these parameters are set as $30$, $2000$, $10^{-8}$ and
$[n^{-\frac{1}{2}},\dots,n^{-\frac{1}{2}}]^T$, respectively.

\section{Numerical examples}\label{sec:5}

In this section, we report numerical experiments on several
matrices to illustrate the performance of our implicitly restarted SSLBD
algorithm with partial reorthogonalization for the
eigenvalue problem of a large skew-symmetric $A$.
The algorithm will be abbreviated as IRSSLBD,
and has been coded in the Matlab language.
All the experiments were performed on an Intel (R)
core (TM) i9-10885H CPU 2.40 GHz with the main memory
64 GB and 16 cores using the Matlab R2021a with the
machine precision $\epsilon=2.22\times 10^{-16}$ under
the Microsoft Windows 10 64-bit system.

\begin{table}[tbhp]
	\caption{Properties of the test matrices
		$A=\frac{A_{\mathrm{o}}-A_{\mathrm{o}}^T}{2}$
		and $A=\bsmallmatrix{&A_\mathrm{o}\\-A_{\mathrm{o}}^T&}$
		with $A_{\mathrm{o}}$
being square and rectangular, respectively, where ``M80PI'',
``visco1'', ``flowm'', ``e40r0'' and  ``Marag'', ``Kemel'',
``storm'', ``dano3'' are abbreviations of ``M80PI\_n1'',
``viscoplastic1'', ``flowmeter5'', ``e40r0100'' and
``Maragal\_6'', ``Kemelmacher'', ``stormg2-27'', ``dano3mip'',
respectively.}\label{table00}
\begin{center}
	\begin{tabular}{cccccccc} \toprule
			{$A_{\mathrm{o}}$}&$n$&$nnz(A)$&$\sigma_{\max}(A)$
			&$\sigma_{\min}(A)$&$\gap(1)$&$\gap(5)$&$\gap(10)$ \\ \midrule
			
			plsk1919	&1919	&9662		&1.71		&3.79e-17	&1.34e-1  &5.17e-3	&5.17e-3  \\
			
			tols4000	&4000	&9154		&1.17e+7	&9.83e-56	&4.97e-3  &4.97e-3  &4.97e-3  \\
			M80PI		&4028	&8066		&2.84		&6.60e-18	&7.03e-3  &3.68e-3  &3.68e-3  \\
			visco1		&4326	&71572		&3.70e+1	&5.42e-21	&1.59e-2  &3.50e-4  &3.50e-4  \\
			utm5940		&5940	&114590		&1.21		&1.83e-7	&5.14e-2  &9.57e-4  &9.57e-4  \\
			coater2		&9540	&264448		&3.11e+2	&1.35e-13	&9.89e-3  &7.38e-3  &4.51e-4  \\
			flowm		&9669	&54130		&3.32e-2	&7.56e-36	&2.87e-3  &8.27e-4  &6.88e-4  \\
			inlet		&11730	&440562		&3.35		&6.17e-7	&1.61e-2  &5.63e-4  &5.63e-4  \\
			e40r0		&17281	&906340		&1.04e+1	&1.54e-16	&9.64e-2  &3.99e-3  &2.00e-3  \\
			ns3Da		&20414	&1660392	&5.12e-1	&2.35e-6	&2.59e-2  &2.38e-3  &8.37e-5  \\
			epb2		&25228	&199152		&1.09		&2.85e-12	&1.15e-2  &8.89e-4  &1.86e-4  \\
			barth		&6691	&39496		&2.23		&5.46e-18	&5.98e-3  &1.85e-5  &1.85e-5  \\
			Hamrle2		&5952	&44282		&5.55e+1	&6.41e-5	&5.27e-6  &5.27e-6  &2.56e-6  \\
						
			delf		&9824  	&30794   	&3.92e+3	&0			&3.12e-3  &3.12e-3  &3.12e-3  \\
			large		&12899 	&41270		&4.04e+3	&0			&2.21e-3  &2.21e-3  &2.21e-3  \\
			Marag		&31407 	&1075388	&1.33e+1	&0			&2.78e-1  &1.66e-2  &7.89e-3  \\
			r05			&14880 	&208290		&1.82e+1	&0			&2.07e-3  &2.07e-3  &5.73e-4  \\
			nl			&22364 	&94070		&2.87e+2	&0			&2.27e-2  &1.38e-3  &7.94e-4  \\
			deter7		&24528 	&74262		&9.67		&0			&3.42e-1  &5.56e-4  &5.56e-4  \\
			Kemel		&38145 	&201750		&2.41e+2	&0			&2.17e-2  &6.96e-3  &5.50e-4 \\			
			p05			&14680 	&118090		&1.28e+1	&0			&5.08e-3  &5.08e-3  &1.93e-4  \\
			storm		&51926 	&188548		&5.45e+2	&0			&1.64e-2  &2.28e-5  &2.28e-5  \\
			south31		&54746 	&224796		&1.29e+4	&0			&6.86e+1  &3.96e-1  &2.02e-2  \\
			dano3		&19053 	&163266  	&1.82e+3	&0			&1.17e-3  &2.31e-6  &2.31e-6  \\
			\bottomrule
	\end{tabular}
\end{center}
\end{table}

Table~\ref{table00} lists some basic properties of the test skew-symmetric
matrices $A=\frac{A_{\mathrm{o}}-A_{\mathrm{o}}^T}{2}$ and
$A=\bsmallmatrix{&A_\mathrm{o}\\-A_{\mathrm{o}}^T&}$
with $A_{\mathrm{o}}$ being the square and rectangular
real matrices selected from
the University of Florida Sparse Matrix Collection
\cite{davis2011university}, respectively,
where $nnz(A)$ denotes the number of nonzero elements of $A$,
$\sigma_{\max}(A)$ and $\sigma_{\min}(A)$ are the largest and smallest
singular values of $A$, respectively,
and $\gap(k):=\min\limits_{1\leq j\leq k}
|\lambda_j^2-\lambda_{j+1}^2|/|\lambda_{j+1}^2-\lambda_{\ell}^2|=
\min\limits_{1\leq j\leq k}|\sigma_j^2
-\sigma_{j+1}^2|/|\sigma_{j+1}^2-\sigma_{\min}^2(A)|$.
All the singular values of $A$ are computed, for the experimental purpose,
by using the Matlab built-in function {\sf svd}.
As Theorem~\ref{thm4} shows, the size of $\gap(k)$ critically affects the
performance of IRSSLBD for computing the $k$
largest conjugate eigenpairs of $A$; the bigger $\gap(k)$ is,
the faster IRSSLBD converges. Otherwise, IRSSLBD
may converge slowly.

Unless mentioned otherwise, we always run the IRSSLBD algorithm
with the default parameters described in the end of Section~\ref{sec:4}.
We also test the Matlab built-in functions {\sf eigs}
and {\sf svds} in order to illustrate the efficiency and reliability
of IRSSLBD.  Concretely, with the
same parameters as those in IRSSLBD, we use {\sf eigs}
to compute $2k$ eigenpairs corresponding to the $2k$ largest conjugate
eigenvalues in magnitude and the corresponding eigenvectors of $A$, and use
{\sf svds} to compute the $k$ or $2k$ largest singular triplets of $A$ to
obtain the $2k$ largest eigenpairs of $A$.
We record the number of matrix-vector multiplications,
abbreviated as $\#{\rm Mv}$, and the CPU time in second, denoted by
$T_{\rm cpu}$,
that the three algorithms use to achieve the same stopping tolerance.
We mention that, for most of the test matrices,
the CPU time used by each of the three algorithms is too short, e.g., fewer
than 0.1s, so that it is hard to make a convincing and
reliable comparison using CPU time.
Therefore, we will not report the CPU time in most of the experiments.
We should remind that a comparison of CPU time used by our algorithm and {\sf
eigs}, {\sf svds} may also be misleading because our IRSSLBD code is the
pure Matlab language while {\sf eigs} and {\sf svds} were programmed using
advanced and much higher efficient C/C++ language.
Nevertheless, we aim to illustrate that IRSSLBD is
indeed fast and not slower than {\sf eigs} and {\sf svds}
even in terms of CPU time.

\begin{exper} \label{exper1}
We compute $k=1,5,10$ pairs of the largest
conjugate eigenvalues in magnitude of
$A=\mathrm{plsk1919}$ and the corresponding eigenvectors.
\end{exper}

\begin{table}[tbhp]
	\caption{Results on $A=\mathrm{plsk1919}$.}\label{table4}
	\begin{center}
		\begin{tabular}{ccccccc} \toprule
			\multirow{2}{*}{\ \ Algorithm\ \ }
			&\multicolumn{2}{c}{$k=1$}
			&\multicolumn{2}{c}{$k=5$}
			&\multicolumn{2}{c}{$k=10$}   \\
			\cmidrule(lr){2-3} \cmidrule(lr){4-5} \cmidrule(lr){6-7}
			&\ \  $\#$Mv\ \  &\ \ \ $T_{\mathrm{cpu}}$\ \ \ \ \
			&\ \  $\#$Mv\ \  &\ \ \ $T_{\mathrm{cpu}}$\ \ \ \ \
			&\ \  $\#$Mv\ \  &\ \ \ $T_{\mathrm{cpu}}$\ \ \ \ \ \\ \midrule
			eigs &58&4.52e-3&122&1.02e-2&186&3.06e-2\\
			SVDS($k$) 	&118&4.11e-3&258&1.22e-2&416&1.61e-2\\
			SVDS($2k$) 	&266&8.35e-3&416&1.61e-2&448&2.04e-2\\
			IRSSLBD 	&50&8.26e-3&106&4.03e-2&134&5.20e-2 \\
			\bottomrule
		\end{tabular}
	\end{center}
\end{table}
Table~\ref{table4} displays the $\#$Mv and $T_{\mathrm{cpu}}$
used by the three algorithms, where
SVDS($k$) and SVDS($2k$) indicate that we use {\sf svds} to compute the $k$ and
$2k$ largest singular triplets of $A$, respectively.
We have found that IRSSLBD and {\sf eigs} solved the SVD problems successfully
and computed the desired eigenpairs of $A$ correctly.
In terms of $\#$Mv, we can see from the table that
IRSSLBD is slightly more efficient than {\sf eigs}
for $k=1,5$ but outperforms
{\sf eigs} substantially for $k=10$. IRSSLBD is thus superior to {\sf eigs}
for this matrix.
We also see from Table~\ref{table4} that the CPU time used by IRSSLBD
is very short and competitive with that used by {\sf eigs}.
The same phenomena have been observed for most of the test matrices.

Remarkably, we have found that, without using
explicit partial reorthogonalization to maintain the
semi-biorthogonality of the left and right Lanczos vectors,
IRSSLBD and {\sf svds} encounter some severe troubles,
and they behave similarly. Indeed, for $k=1$, SVDS($k$)
succeeded in computing the largest singular triplet of $A$,
from which the desired largest conjugate eigenpairs are recovered
as described in Section~\ref{subsec:2}.
For $k=5$ and $10$, however, SVDS($k$) computes duplicate
approximations to the singular triplets associated with each
conjugate eigenvalue pair of $A$, and ghosts appeared,
causing that  SVDS($k$) computed only half of the desired
eigenpairs, i.e., the first to third and the first to fifth
conjugate eigenpairs, respectively.
The ghost phenomena not only delay the convergence of SVDS($k$)
considerably but also make it converge irregularly and fail
to find all the desired singular triplets. SVDS($k$) also
uses much more $\#$Mv than IRSSLBD does for $k=10$.
Likewise, for each $k=1,5,10$, SVDS($2k$) succeeds to compute only $k$
approximations to the desired largest singular triplets of $A$, from which the
desired approximate eigenpairs can be recovered, but consumes
$3\sim 5$ times $\#$Mv as many
as those used by IRSSLBD. This shows that
direct application of an LBD algorithm to
skew-symmetric matrix eigenvalue problems does not work
well in finite precision arithmetic. In contrast, IRSSLBD
works well when the semi-orthogonality and semi-biorthogonality
are maintained. Therefore, the semi-biorthogonality of
left and right Lanczos vectors is crucial.

\begin{exper}\label{exper2}
We compute $k=1,5,10$ pairs of largest conjugate eigenvalues
in magnitude and the corresponding eigenvectors of the
twelve square matrices
$A=\frac{A_{\mathrm{o}}-A_{\mathrm{o}}^T}{2}$ with
$A_{\mathrm{o}}=\mathrm{tols4000}$, $\mathrm{M80PI}$, $\mathrm{visco1}$,
$\mathrm{utm5940}$, $\mathrm{coater2}$, $\mathrm{flowm}$, $\mathrm{inlet}$,
$\mathrm{e40r0}$, $\mathrm{ns3Da}$, $\mathrm{epb2}$, $\mathrm{barth}$ and
$\mathrm{Hamrle2}$, respectively.
We will report the results on IRSSLBD and {\sf eigs} only since
{\sf svds} behaves irregularly
and computes only partial desired singular triplets because of the severe
loss of numerical biorthogonality of left and right Lanczos vectors.  	
\end{exper}

\begin{table}[tbhp]
	\caption{The \#{\rm Mv} used by IRSSLBD and {\sf eigs}
		to compute $k=1,5,10$ pairs of largest conjugate
		eigenvalues in magnitude and the
corresponding eigenvectors of $A=\frac{A_{\mathrm{o}}-A_{\mathrm{o}}^T}{2}$.}\label{table1}
	\begin{center}
		\begin{tabular}{ccccrcc} \toprule
			\multirow{2}{*}{$A_{\mathrm{o}}$}
			&\multicolumn{3}{c}{Algorithm: {\sf eigs}}
			&\multicolumn{3}{c}{Algorithm: IRSSLBD}
			   \\
			\cmidrule(lr){2-4} \cmidrule(lr){5-7}
			&\ \ $k=1$\ \ &\ \ $k=5$\ \  &\ \ $k=10$\ \
			&\ $k=1$\ \ \ \ \ \ &\ $k=5$\ \  &\ $k=10$ \ \\ \midrule
			
			tols4000	&338	&278	&460 &174 (51.48)	&232 (83.45) 	&292 (63.48)	    \\
			M80PI		&58		&110	&254 &44  (75.86)	&100 (90.91)	&158 (62.20)		\\
			visco1		&198	&302	&912 &132 (66.67)	&282 (93.38)	&370 (40.57)		\\
			utm5940		&86		&336	&788 &76  (88.37)	&278 (82.74)	&306 (38.83)		\\		
			coater2		&142	&162	&282 &84  (59.15)	&130 (80.25)	&152 (53.90)	    \\
			flowm		&478	&384	&702 &228 (47.70)	&346 (90.10)	&394 (56.13)		\\
			inlet		&170	&412	&572 &132 (77.65)	&316 (76.70)	&310 (54.20)		\\
			e40r0		&86		&168	&576 &70  (81.40)	&148 (88.10)	&284 (49.31)		\\
			ns3Da		&114	&226	&960 &92  (80.70)	&204 (90.27)	&376 (39.17)		\\
			epb2 		&198	&328	&828 &138 (69.70)	&278 (84.76)	&384 (46.38)		\\
			barth		&310	&392	&510 &172 (55.48)	&318 (81.12)	&322 (63.14)		\\
			Hamrle2		&3558	&300  &16178 &136 (\hspace{0.48em}3.82)	&190 (63.33)	&354 (\hspace{0.48em}2.19)		\\			
			\bottomrule
		\end{tabular}
	\end{center}
\end{table}

IRSSLBD and {\sf eigs} successfully compute the desired eigenpairs
of all the test matrices.
Table~\ref{table1} displays the $\#$Mv used by them,
where each quantity in the parenthesis is the percentage of
$\#$Mv used by IRSSLBD over that by {\sf eigs} and we drop
$\%$ to save the space.

For $k=1$, we see from Table~\ref{table1} that
IRSSLBD is always more efficient than {\sf eigs}, and it often outperforms
the latter considerably for most of the test matrices.
For instance, IRSSLBD uses fewer than $60\%$ $\#$Mv of {\sf eigs}
for $A_{\mathrm{o}}=$ tols4000, coater2, flowm and barth.
Strikingly, IRSSLBD consumes only $3.82\%$ of the $\#$Mv used
by {\sf eigs} for $A_{\mathrm{o}}=\mathrm{Hamrle2}$.
For $k=5$, IRSSLBD is often considerably more efficient than {\sf eigs}
for most of the test matrices, especially $A_{\mathrm{o}}=$ Hamrle2,
where IRSSLBD uses $36\%$ fewer $\#$Mv.
For $k=10$, IRSSLBD uses $62\%$--$64\%$ of the $\#$Mv consumed by {\sf eigs}
for $A_{\mathrm{o}}=$ tols4000, M80PI and barth,
and it uses $54\%$--$57\%$ of the $\#$Mv
for $A_{\mathrm{o}}=$ coater2, flowm, and inlet.
Therefore, for these six matrices, the advantage of IRSSLBD over {\sf eigs}
is considerable. For the five matrices
$A_{\mathrm{o}}=$ visco1, utm5940, e40r0, ns3Da and epb2,
IRSSLBD is even more efficient than ${\sf eigs}$
as it consumes no more than half
of the $\#$Mv used by {\sf eigs}. For $A_{\mathrm{o}}=$ Hamrle2,
IRSSLBD only uses $2.19\%$ of the $\#$Mv cost by {\sf eigs}, and the
save is huge.

In summary, for these twelve test matrices, our IRSSLBD algorithm outperforms
{\sf eigs} and is often considerably more efficient than the
latter for the given three $k$.

\begin{exper}
We compute $k=1,5,10$ pairs of the largest conjugate eigenvalues of
$A=\bsmallmatrix{&A_{\mathrm{o}}\\-A_{\mathrm{o}}^T&}$ and the corresponding
eigenvectors with the eleven matrices
$A_{\mathrm{o}}=$ delf, large, Marag, r05, nl,
deter7, Kemel, p05, storm, south31 and dano3.
\end{exper}

\begin{table}[tbhp]
	\caption{The \#Mv used by IRSSLBD and {\sf eigs}
to compute $k=1,5,10$ pairs of the largest conjugate eigenpairs of
$A=\bsmallmatrix{&A_{\mathrm{o}}\\-A_{\mathrm{o}}^T&}$. }\label{table3}
	\begin{center}
		\begin{tabular}{ccccrrr} \toprule
			\multirow{2}{*}{$A_{\mathrm{o}}$}
			&\multicolumn{3}{c}{Algorithm: {\sf eigs}}
			&\multicolumn{3}{c}{Algorithm: IRSSLBD}
			   \\
			\cmidrule(lr){2-4} \cmidrule(lr){5-7}
			&\ $k=1$\  &\ $k=5$\   &\ $k=10$\ \
			&\ $k=1$\ \ \ \ \ \ &\ \ $k=5$\ \ \ \ \ \  &\ \ $k=10$\ \ \ \ \ \ \\ \midrule			
			delf		&478	&214	&192 &142 (29.71)	    &142 (66.36)	&138 (71.88)	  \\
			large		&590	&284	&268 &188 (31.86)	    &172 (60.56)	&172 (64.18)	  \\
			Marag		&58		&84		&178 &38 (65.52)		&82 (97.62)		&130 (73.03)	  \\
			r05			&86		&70		&668 &54 (62.79)		&56 (80.00)		&222 (33.23)	  \\
			nl			&114	&242	&222 &76 (66.67)		&172 (71.07)	&126 (56.76)	  \\
			deter7		&58		&346	&144 &36 (62.07)		&148 (42.77)	&122 (84.72)	  \\
			Kemel		&142	&220	&636 &118 (83.10)	    &198 (90.00)	&306 (48.11)      \\
			p05			&114	&90		&692 &58 (50.88)		&62 (68.89)		&234 (32.82)	  \\
			storm		&58		&84		&78  &46 (79.31)		&64 (76.19)		&66  (84.62)	  \\
			south31		&30		&42		&146 &8  (26.67)		&26 (61.90)		&92  (63.01)	  \\
			dano3		&114	&70	   &5922 &58 (50.88)		&70 (100.0)		&196 (\hspace{0.45em}3.31)	  \\
			
			\bottomrule
		\end{tabular}
	\end{center}
\end{table}

These test matrices are all singular with large dimensional null spaces.
In order to purge the null space from the searching subspaces,
we take the initial vector $q_{1}$ in IRSSLBD and {\sf eigs} to
be the unit-length vector normalized from $A\cdot [1,\dots,1]^T$.
All the other parameters are the same for the two algorithms with the
default values as described previously.
Both IRSSLBD and {\sf eigs} succeed in computing all the desired
eigenpairs of $A$ for the three $k$'s.
Table~\ref{table3} displays the \#Mv, where the
quantities in the parentheses are as in Table~\ref{table1}.

We have also observed an interesting phenomenon from this table
and the previous experiments. As is seen from Table~\ref{table3},
for some test matrices, e.g., delf, large, nl, deter7, p05 and dano,
when more conjugate eigenpairs of $A$ are desired,
IRSSLBD or {\sf eigs} may use even fewer $\#$Mv.
This occurs mainly due to the implicit restart technique.
Specifically, when more conjugate eigenpairs are desired,
larger dimensional new initial searching subspaces $\UU_k$ and $\VV_k$ are
retained in the next restart, so that
the resulting restarted searching subspaces may contain more information
on the desired eigenvectors, causing that
the largest approximate eigenpairs may converge faster.

As for an overall efficiency comparison of IRSSLBD and {\sf eigs},
we have also observed the advantage
of the former over the latter, similar to that in Experiment~\ref{exper2}.
Therefore, we now make some comments on them together.
Before doing so, let us regard each test matrix with each selected
$k$ as one independent problem, so in total we have 33 problems, which
will be distinguished by the name-value pairs $(A_{\mathrm{o}}, k)$.

As is seen from Table~\ref{table3}, for
$(A_{\mathrm{o}},k)=(\mathrm{Marag},5)$ and $(\mathrm{dano3},5)$,
IRSSLBD and {\sf eigs} use almost the same $\#$Mv.
For $(A_{\mathrm{o}},k)=(\mathrm{Kemel},1)$, $(\mathrm{Kemel},5)$,
$(\mathrm{deter7},10)$ and  $(\mathrm{strom},10)$,
IRSSLBD is slightly better than {\sf eigs} and uses $82\%$--$90\%$ of
the $\#$Mv.
For $(A_{\mathrm{o}},\!k)\!=\!(\mathrm{Marag},\!1)$,
$(\mathrm{r05},1)$, $(\mathrm{nl},1)$,
$(\mathrm{deter7},1)$,
$(\mathrm{storm},1)$,
$(\mathrm{delf},5)$, $(\mathrm{large},\! 5)$, $(\mathrm{r05},5)$,
$(\mathrm{nl},\!5)$, $(\mathrm{p05},5)$, $(\mathrm{storm},\! 5)$,
$(\mathrm{south31},5)$, $(\mathrm{delf},10)$, $(\mathrm{large},10)$,
$(\mathrm{Marag},10)$, $(\mathrm{nl},10)$ and
$(\mathrm{south31},10)$, IRSSLBD outperforms {\sf eigs} considerably
as it costs $56\%$--$80\%$ of the $\#$Mv used by {\sf eigs}.
For $(A_{\mathrm{o}},k)=(\mathrm{delf},1)$,
$(\mathrm{large},1)$, $(\mathrm{p05},1)$,
$(\mathrm{south31},1)$, $(\mathrm{dano3},1)$,
$(\mathrm{deter7},5)$, $(\mathrm{r05},10)$,
$(\mathrm{Kemel},10)$,
$(\mathrm{p05},10)$
and $(\mathrm{dano3},10)$,
IRSSLBD is substantially more efficient than
{\sf eigs} as it uses only approximately
half of the $\#$Mv or even fewer.
Particularly, for $(A_{\mathrm{o}},k)=(\mathrm{dano3},10)$,
the advantage of IRSSLBD over {\sf eigs} is very substantial since,
$\#$Mv used by the former is only $3.5\%$ of that by the latter.

In summary, by choosing an appropriate starting vector,
IRSSLBD suits well for computing several largest conjugate eigenpairs
of a singular skew-symmetric matrix, and it is more and
can be much more efficient than {\sf eigs}.

\section{Conclusions}\label{sec:6}

We have shown that the eigendecomposition of a real skew-symmetric matrix
$A$ has a close relationship with its structured SVD,
by which the computation of its several extreme conjugate eigenpairs
can be equivalently transformed into the computation of the largest
singular triplets in real arithmetic. For $A$ large,
as a key step toward the efficient computation of the desired
eigenpairs, by means of
the equivalence result on the tridiagonal decomposition of $A$
and a half sized bidiagonal decomposition, the SSLBD process is exploited to
successively compute a sequence of partial bidiagonal decompositions. The
process provides us with a sequence of left and right searching subspaces and
bidiagonal matrices, making
the computation of a partial SVD of $A$ possible,
from which the desired eigenpairs of $A$
is obtained without involving complex arithmetic.

We have made a detailed theoretical analysis
on the SSLBD process. Based on the results
obtained, we have proposed a SSLBD method for
computing several extreme singular triplets of $A$, from which the desired
extreme conjugate eigenvalues in magnitude and the corresponding
eigenvectors can be recovered.
We have established estimates for the distance between a desired
eigenspace and the underlying subspaces that the SSLBD
process generates, showing how it converges to zero.
In terms of the distance between the desired eigenspace and the searching
subspaces, we have derived a priori error bounds for the recovered conjugate
eigenvalues and the associated eigenspaces. The results indicate
that the SSLBD method generally favors extreme
conjugate eigenpairs of $A$.

We have made a numerical analysis on the SSLBD process and
proposed an efficient and reliable approach to
track the orthogonality and biorthogonality among the computed left and
right Lanczos vectors. Unlike the standard LBD process, for
its skew-symmetric variant SSLBD, we have proved that
only the semi-orthogonality of the left and right Lanczos vectors
does not suffice and that the semi-biorthogonality of
these two sets of vectors is absolutely necessary, which has been
numerically confirmed.
Based on these, we have designed an effective partial reorthogonalization
strategy for the SSLBD process
to maintain the desired semi-orthogonality and semi-biorthogonality.
Combining the implicit restart with the partial reorthogonalization proposed,
we have developed an implicitly restarted SSLBD algorithm to
compute several largest singular triplets of a
large real skew-symmetric matrix.

A lot of numerical experiments have confirmed the effectiveness and
high efficiency of the implicitly restarted SSLBD algorithm. They
have indicated that it is at
least competitive with {\sf eigs} and often outperforms
the latter considerably.


\end{document}